\documentclass[a4paper,11pt]{amsart}

\usepackage[british]{babel}
\usepackage[colorlinks=true,citecolor=black,urlcolor=black,linkcolor=black]{hyperref}
\usepackage{enumitem}
\usepackage{amssymb}
\usepackage{microtype}

\numberwithin{equation}{section}
\newtheorem{theorem}{Theorem}
\newtheorem{lemma}{Lemma}[section]
\newtheorem{proposition}[lemma]{Proposition}
\theoremstyle{definition}
\newtheorem{definition}[lemma]{Definition}
\theoremstyle{remark}
\newtheorem{remark}[lemma]{Remark}
\newtheorem{question}{Question}

\newcommand{\cA}{\mathcal{A}} 
\newcommand{\cC}{\mathcal{C}} 
\newcommand{\cF}{\mathcal{F}} 
\newcommand{\cG}{\mathcal{G}} 
\newcommand{\cL}{\mathcal{L}} 
\newcommand{\cM}{\mathcal{M}} 
\newcommand{\cT}{\mathcal{T}} 
\newcommand{\cK}{\mathcal{K}} 
\newcommand{\bC}{\mathbb{C}} 
\newcommand{\bN}{\mathbb{N}} 
\newcommand{\bR}{\mathbb{R}} 
\newcommand{\bT}{\mathbb{T}} 
\newcommand{\bZ}{\mathbb{Z}} 
\newcommand{\fA}{\mathfrak{A}} 
\newcommand{\fB}{\mathfrak{B}} 
\newcommand{\abs}[1]{\left|#1\right|} 
\newcommand{\norm}[1]{{\left\|#1\right\|}} 
\newcommand{\Jac}[2]{\operatorname{Jac}({#2|}_{#1})} 
\newcommand{\bigG}{\widehat{\Gamma}} 
\newcommand{\len}[1]{\abs{#1}} 
\newcommand{\Jmp}{J} 
\newcommand{\gamr}{\gamma} 
\newcommand{\gaml}{\gamma^{\prime}} 
\newcommand{\doi}[1]{}

\title{Discontinuities cause essential spectrum on surfaces}

\author[O. Butterley]{Oliver Butterley}
\address{(Oliver Butterley) Department of Mathematics, University of Rome Tor Vergata -- Via della Ricerca Scientifica 1 -- 00133 Roma -- Italy}
\email{butterley@mat.uniroma2.it}

\author[G. Canestrari]{Giovanni Canestrari}
\address{(Giovanni Canestrari) Department of Mathematics, University of Rome Tor Vergata -- Via della Ricerca Scientifica 1 -- 00133 Roma -- Italy}
\email{canestrari@mat.uniroma2.it}

\author[R. Castorrini]{Roberto Castorrini}
\address{(Roberto Castorrini) Dipartimento di Matematica, Universit\`a di Pisa -- Largo B.~Pontecorvo 5 -- 56127 Pisa --  Italy}
\email{roberto.castorrini@gmail.com}

\date{}

\begin{document}

\begin{abstract}
    Two-dimensional maps with discontinuities are considered. 
    It is shown that, in the presence of discontinuities, the essential spectrum of the transfer operator is large whenever it acts on a Banach space with norm that is stronger than \(L^\infty\) or \(BV\). 
    Three classes of examples are introduced and studied, both expanding and partially expanding.
    In two dimensions there is complication due to the geometry of the discontinuities, an issue not present in the one-dimensional case and which is explored in this work.
\end{abstract}

\maketitle
\thispagestyle{empty}

\section{Introduction}
\label{sec:intro}

Dynamical systems with some degree of hyperbolicity are typically known to behave with good statistical properties, e.g., exponential decay of correlations with respect to a physical measure, linear response, etc.
A particularly strong property for the system is to exhibit a full resonance spectrum, in the sense that correlations can be described to arbitrary precision (see, e.g.,~\cite{FGL19,BS20,BKL22} and references within). 
However this is not always possible.
Systems with discontinuities often occur in physically motivated examples, e.g., in billiard flows due to glancing collisions and in the Lorenz flow due to a shearing behaviour at a singular point.
To what extent can discontinuities show an effect in the statistical properties? 

The subject of this present work is to demonstrate that discontinuities really do affect the statistical properties of a dynamical system, besides causing countless technical problems in proofs.
In one-dimensional systems there is a rather complete description of how discontinuities cause essential spectrum~\cite{BCJ22}.
However the restriction on dimension is unfortunate in view of application to general systems and so here we take the step of extending to higher dimension.
For simplicity we consider the two-dimensional case, but with the aim of developing the tools for higher dimensions.
We work with expanding or partially expanding systems, leaving aside for moment the equally relevant hyperbolic case. 

The question of resonance spectrum is closely related to the one of determining the spectrum of the transfer operator acting on some Banach space.
Specifically, it is useful to know that the essential spectrum is small since the spectrum outside of the essential spectrum consists of isolated eigenvalues of finite multiplicity. 
In such a way the ``true" point spectrum is revealed and not the possible ``fake" part due to the choice of the particular Banach space.
In general the spectrum of an operator depends on the choice of Banach space but, given the statistical properties motivations, it is natural to require the space to contain all smooth observables.
On the other hand, the Banach space can't be too small because then the spectrum might be missing some information about the observables of interest.
The present question is closely related to zeta functions, a notion of describing resonances inherent to a system (see, e.g., \cite{Baladi18}). 
The analogous result in this context would be showing the impossibility of meromorphically extending the zeta function to the whole complex plane.

In the case of one-dimensional expanding interval maps it was shown~\cite{BCJ22} that there is a dichotomy, either the system is Markov or there exists an unavoidable disc of essential spectrum with radius dependent on the asymptotic behaviour of the weight of the chosen transfer operator.
In the two-dimensional case we haven't been able to conclude such a simple story. 
One issue is that, in higher dimensions, there are multiple different ways a system can fail to be Markov (see further discussion in Section~\ref{sec:discuss}).
This issue is intimately connected with the fact that discontinuities in two dimensions (and higher) can exhibit subtle and complex behaviour (see, e.g., \cite{Buzzi00,Buzzi01,Tsujii00}). 
Consequently it is much harder to work with the discontinuity set compared to the one-dimensional case where the singularities are isolated points and there is no geometry to consider.
Indeed, it is well established that controlling ``complexity'' of the discontinuity is essential for working with such systems (see, e.g., \cite{Saussol00} and \cite[Chapter 5]{CM06} concerning controlling complexity and demonstrating that growth dominates cutting).

In the one-dimensional case~\cite{BCJ22} there were some restrictions, due to the method, on the Banach space for which the result applies.
To be precise, the result held only for Banach spaces which are stronger than \(L^\infty\).
This restriction is due to the way the technique involves the behaviour in the dual: loosely speaking, if the space is too large then the dual is too small to complete the argument.
However, these authors believe that this is an artificial requirement since the claimed result is known to also hold for all other Banach spaces which they are aware of~\cite[\S2]{BCJ22}.
In two dimensions the reference norm is less obvious and the results here are presented for two separate assumptions, namely \(L^\infty\) and \(BV\).
Neither of these spaces is stronger than the other in dimension two or higher and so the estimates which follow are, depending on the system being studied, sometimes better in one case, sometimes in the other.

In this text we present the main result (Theorem~\ref{thm:main}) and three distinct example classes of system for which this applies.
The first example (Section~\ref{sec:affinemap}) is a class of piecewise affine expanding maps of the torus where the discontinuity propagates along lines which, iterate after iterate, change in slope and so never overlap with each other.
The second class of examples (Section~\ref{sec:cocycle}) consists of maps of the torus which can be written as cocycles.
Here we see that the non-Markovness of the base map leads to a discontinuity of the two-dimensional system and hence unavoidable essential spectrum.
The third class of examples (Section~\ref{sec:local-example}) are expanding maps of the torus which are smooth except for in a very small neighbourhood. 
In this way we see that a local discontinuity of the system is sufficient to lead to large essential spectrum.
In the presentation of the main argument we take this opportunity to isolate the abstract functional analytic result (Proposition~\ref{prop:abs}) which is at the core of this work and which was also present in the one-dimensional case yet slightly obscured.

Continuing forward, there are the following pertinent questions.

\begin{question}[dichotomy for the 2D case]
    Can a lower bound on the essential spectral radius be proven for any two-dimensional piecewise expanding map which fails to be Markov?
    The assumption used in this work is arguably natural and verifiable.
    Nonetheless it would be tidy to either have no additional assumption or to show the existence of obstacles to such a relaxation of the assumptions (see Section~\ref{sec:discuss}).
\end{question}

\begin{question}[hyperbolic case]
    Can these ideas be extended to the hyperbolic case? 
    Specifically, consider the two-dimensional hyperbolic map associated to Sinai billiards.
    Can one prove a lower bound for the essential spectral radius of the transfer operator associated to the SRB measure for any useful Banach space?
    Such a study would need to consider Banach spaces of distributions, specifically the anisotropic spaces suitable for studying hyperbolic systems (see Remark \ref{rem:compare}).
\end{question}

\begin{question}[diverse notions of resonances]
    Recall that there are several closely related notions:
    \begin{itemize}
        \item Obstacles to the meromorphic extension of the zeta function,
        \item Precision to which it is possible to describe correlations in terms of resonances for \(\cC^\infty\) observables,
        \item Lower bound on essential spectral radius of the transfer operator for any useable Banach space.
    \end{itemize}
    Is it possible to formalize completely the connection between these notions? 
\end{question}

\begin{question}[arbitrary dimension]
    Can these ideas be extended to arbitrary dimension for expanding systems?
    The structure of this present argument suggests that such would be possible in a similar way but using codimension one linear functionals.
\end{question}

Section~\ref{sec:results} contains the key definitions and our results.
Section~\ref{sec:discuss} is devoted to discussion of the assumptions on the discontinuities.
The proof of an abstract functional analytic result is contained in Section~\ref{sec:ess-spec}.
This is a result which relates the existence of a family of linear functionals to a lower estimate of the essential spectral radius and so, in Section~\ref{sec:lin-funcs} we construct the relevant linear functionals. 
Section~\ref{sec:examples} contains a description and the results related to three example classes of systems.

\section*{Acknowledgements}

\begin{small}
    We are grateful to Péter Bálint, Bernardo Carvalho, Daniele Galli, Marco Lenci, Carlangelo Liverani and Daniel Smania for several enlightening discussions. 
    We also thank the anonymous referee for their useful comments and suggestions. 
    This research is part of the authors' activity within the the UMI Group ``DinAmicI'' and the INdAM group GNFM.
    The authors are partially supported by PRIN Grant ``Regular and stochastic behaviour in dynamical systems" (PRIN 2017S35EHN). 
    O.B. and G.C. acknowledge the MIUR Excellence Department Projects awarded to the Department of Mathematics, University of Rome Tor Vergata, CUP E83C18000100006, CUP E83C23000330006.
\end{small}

\section{Settings and results}
\label{sec:results}

Let $\cM$ be a compact, smooth and connected Riemannian manifold of dimension $d=2$. 
We suppose that \(\cF : \cM \to \cM\) is a \(\cC^r\) \emph{piecewise non-singular map}, \(r \ge 1\), in the sense that \(\cF\) is a \(\cC^r\) local diffeomorphism on \(\cM \setminus \Gamma\) where \(\Gamma\) is a finite union of finite length \(\cC^1\) curves and the derivative map is uniformly non-singular.\footnote{We write \emph{uniformly non-singular} to mean that the determinant of the derivative map (in some coordinate atlas) is uniformly bounded away from zero.}
For convenience we suppose that \(\cF\) is orientation preserving.

Let \(\varphi: \mathcal{M} \rightarrow \bC\) be \(\cC^1\) on \(\cM \setminus \Gamma\), uniformly bounded and uniformly bounded away from zero. %
We call \(\varphi\) the \emph{weight} of the transfer operator.
The transfer operator is defined, pointwise for any \(h: \cM \rightarrow \bC\), as
\begin{equation}
    \label{eq:def-L}
    \cL h(y)= \sum_{\cF x=y} (h \cdot \varphi)(x).
\end{equation}

Keeping track of the discontinuity set will be crucial and so we introduce the following notation.
Let \(M\) denote the closure of \(\cM \setminus \Gamma\) in the sense that \(M\) is a smooth Riemannian manifold with boundary and the boundary comprises of a finite union of \(\cC^1\) curves originating from \(\Gamma\). 
Call \(F : M \to \cM\) the unique \(\cC^r\) extension of \(\cF\) to \(M\).
In this way we have a solid way to work with the limits from one side and the other of \(\Gamma\). 
Let \(\iota : M \to \cM\) denote the embedding between manifolds and observe that it is 1-to-1 except on \(\partial M\) where\footnote{Intersections of curves in \(\Gamma\) (at most a finite number) cause \(\iota\) to be many-to-1.} it is mostly 2-to-1.
Let,
\[
    \Gamma_1 = F\Gamma, \quad\quad 
    \Gamma_{j+1} = \cF^j \Gamma_{1}, \ j \in \bN, \quad\quad 
    \bigG = \bigcup\nolimits_{j\in\bN}\Gamma_j \subset \cM.
\]
We anticipate here that the transfer operator \(\cL\) \eqref{eq:def-L} involves the inverse of \(\cF\) and so the location of discontinuities of the images of a smooth observable under iterates of \(\cL\) is \(\widehat \Gamma\) (see Lemma~\ref{lem:struct-of-functions}).

If \(A,B\) are two sets in \(M\), we say that they are \emph{equivalent}, and we write \(A \approxeq B\), if \(A\) and \(B\) are equal modulo a set of points whose \(1\)-dimensional Hausdorff measure is zero.
For any curve \(\gamma \subset \cM\) and map \(\cG : \cM \to \cM\), let \(\Jac{\smash{\gamma}}{\cG}\) denote the \emph{Jacobian} of \(\cG\) restricted to the curve \(\gamma\).\footnote{In particular this means that, if \(\cG\) restricted to \(\gamma\) is measurable and injective, then, for any \(\phi : \gamma\to \bR\), \(\int_{\cG \gamma} \phi(y) \ dy = \int_{\gamma} \phi(\cG x) \Jac{\gamma}{\cG}(x) \ dx\).}

\begin{definition}[proper discontinuity]
    \label{def:proper}
    Given $\cF$ as above, we say that a $\cC^1$ connected curve \(\gamma \subset \partial M\) is a \emph{proper discontinuity} if, denoting\footnote{Recall that $F$ is the extension of $\cF$ to $M$.} \(\gamma_1 =F\gamma\), \(\gamma_{k+1} = \cF^{k}\gamma_1\), \(k \in \bN\), the following are satisfied:
    \begin{enumerate}[label=(A-\arabic*),labelsep=1em,labelindent=!,start=0]
        \item \label{it:A0} 
        \(\cF\) is not continuous at \(x\) for all \(x\in\iota\gamma\);
        \item \label{it:A1}
        There exist functions, \(\alpha_k : \gamma_k \to \bC\), \(k\in\bN\) such that \(\alpha_{1} \equiv 1\) and
            \begin{equation}
                \label{eq:good-weights}
                \alpha_{k+1}(\cF x) = {\alpha_k(x)}\varphi(x)^{-1} \Jac{\gamma_k}{\cF}(x)^{-1} \quad \text{for a.e. \(x \in \gamma_{k}\);}\footnote{I.e., for all \(x\in \tilde\gamma_k\) where \(\tilde\gamma_k \approxeq \gamma_k\).}
            \end{equation}
        \item \label{it:A2} 
        \(\gamma_k \cap \Gamma_1 \approxeq \emptyset \) for all \(k\geq 2\);
        \item \label{it:A3} 
        \(\cF^{-1}\gamma_{k+1} \cap \bigG\approxeq \gamma_{k}\) for all \(k\in\bN\).
    \end{enumerate}
\end{definition}
Notice that the right hand side of \eqref{eq:good-weights} is well defined since \(\varphi(x)\) is uniformly bounded away from zero and \(\cF\) is non-singular.
We denote by $BV(\cM)$ the usual Banach space of functions of bounded variation on $\cM$.
For convenience we use the notation \(\varphi_n = \prod_{k=0}^{n-1} \varphi\circ \cF^k\).
Given a map $\cF: \cM\to \cM$ and transfer operator weight $\varphi:\cM \to \bC$ as above, if \(\gamma\) is a proper discontinuity, we consider the following quantities:\footnote{In the appropriate context, as here, \(\len{\cdot}\) denotes the length of the curve.}
\begin{equation}
    \label{eq:Lambda}
    \begin{aligned}
        \Lambda_{\operatorname{BV}}\left(\cF, \varphi, \gamma\right)
         & =\liminf_{k \rightarrow \infty} \left(\inf_{x \in \gamma} \left| \varphi_k(x) \cdot \Jac{\gamma}{\cF^k}(x) \right|\right)^{\frac{1}{k}},                           \\
        \Lambda_{L^{\infty},1}\left(\cF, \varphi, \gamma\right)
         & = \liminf_{k \rightarrow \infty} \left( \inf_{x \in \gamma} \left|\varphi_k(x)  \right| \right)^{\frac{1}{k}},                                                      \\
        \Lambda_{L^{\infty},2}\left(\cF, \varphi, \gamma\right)
         & = \liminf_{k \rightarrow \infty} \left(\inf_{x \in \gamma} \left| \varphi_k(x) \cdot \Jac{\gamma}{\cF^k}(x) \right| \len{\gamma_{k+1}}^{-1} \right)^{\frac{1}{k}}, \\
        \Lambda_{L^{\infty}}
         & = \max \left(\Lambda_{L^{\infty},1}, \Lambda_{L^{\infty},2}\right).
    \end{aligned}
\end{equation}
In different settings one or the other will provide a better bound (see Remark~\ref{rem:compare}).

Our main result is the following.
\begin{theorem}
    \label{thm:main}
    Let \(\cL\) be the transfer operator associated to a piecewise non-singular map \(\cF:\cM\to \cM\) and weight \(\varphi\) as above. 
    Suppose that $\cF$ admits a proper discontinuity $\gamma$
    and that the Banach space \((\fB,{\norm{\cdot}}_{\fB})\) satisfies
    \begin{itemize}
        \item \(\cC^{\infty}(\cM) \subset \mathfrak{B}\);
        \item \(\cL\) extends to a continuous operator on \(\mathfrak{B}\).
    \end{itemize}
    Then the essential spectral radius of \(\cL : \fB \to \fB\) is at least
    \begin{equation}
        \label{eq:defLambda}
        \begin{aligned}
            \Lambda_{BV}(\cF,\varphi, \gamma)    \quad      & \text{if \({\|h\|}_{BV} \le {\|h\|}_{\fB}\),  \(\forall h \in \fB\)},         \\
            \Lambda_{L^{\infty}}(\cF,\varphi, \gamma) \quad & \text{if  \({\|h\|}_{L^{\infty}} \le {\|h\|}_{\fB}\), \(\forall h \in \fB\)}.
        \end{aligned}
    \end{equation}
\end{theorem}
\noindent
In Section \ref{sec:examples} this result is applied to three classes of systems and the usefulness of the different bounds is discussed.

The above theorem follows from the following two propositions. The first is an abstract functional analytic result which uses the existence of a family of linear functionals to obtain a lower bound of the essential spectral radius for a linear operator.
\begin{proposition}
    \label{prop:abs}
    Let \(L\) be a bounded linear operator on a Banach space \((\fB,\norm{\cdot})\).
    Assume that there exist \(h_{0}\in \fB\), \(\Lambda>0\) and a sequence of linear functionals \({\{\ell_{k}\}}_{k=1}^{\infty}\), satisfying the following:%
    \footnote{We use the standard notation for the dual: if \(h\in\fB\) and \(\ell \in \fB^*\) is an element of the dual, i.e., a linear functional, the dual operator is \(\left(L^* \ell\right)(h) = \ell(Lh)\) and the dual norm is \(\norm{\ell} = \sup \left\{\abs{\ell(h)} : h \in \fB, \norm{h} \leq 1\right\}\).}
    \begin{enumerate}[label=(B-\arabic*),labelsep=1em,labelindent=!]
        \item 
              \label{it:B1}
             \(\limsup_{k\to\infty} \norm{\ell_{k}}^{\frac{1}{k}} \leq \Lambda^{-1}\);
        \item 
              \label{it:B2}
              \(\ell_{1}(h_{0}) = 1\) and, for all \(k\in\bN\),
              \(\ell_{k+1}(h_{0}) = 0\);
        \item 
              \label{it:B3}
              For all \(k\in\bN\), 
              \(L^{*}\ell_{k+1} = \ell_{k}\).           
    \end{enumerate}
    Then the essential spectral radius of \(L : \fB \to \fB\) is at least \(\Lambda\).
\end{proposition}
\noindent
Section~\ref{sec:ess-spec} is devoted to the proof of the above.

At the very least we would like to consider the set of observables, \(\cC^{\infty}(\cM)\).
However \(\cL\) might not leave \(\cC^{\infty}(\cM)\) invariant, particularly in the discontinuous case of present interest.
Let \(\fA = \mathrm{span}\left({\{\cL^n (\cC^{\infty}(\cM))\}}_{n=0}^{\infty} \right)\). The following result uses the existence of a proper discontinuity to construct a family of linear maps which satisfy the assumptions of the above proposition. 
\begin{proposition}
    \label{prop:funct}
    Let \(\cL\) be the transfer operator associated to a piecewise non-singular map \(\cF:\cM\to \cM\) and weight \(\varphi\) as above. 
    Suppose that $\cF$ admits a proper discontinuity $\gamma$. 
    Then there exists a sequence of linear maps \({\{\ell_{k}: \fA \to \bC\}}_{k\in\bN}\) and \(h_0 \in \fA\) which satisfy:%
    \footnote{Where \(\Lambda_{L^\infty} = \Lambda_{L^\infty}(\cF,\varphi, \gamma)\) and \(\Lambda_{BV} = \Lambda_{BV}(\cF,\varphi, \gamma)\) are as defined in \eqref{eq:Lambda}. When this does not create confusion, subsequently we may suppress the dependency on \((\cF,\varphi, \gamma)\).}
    \begin{equation}
        \label{eq:two-bounds}
        \begin{aligned}
            \limsup_{k\to\infty} \Big( \sup \left\{ \abs{\ell_{k}(h)} : {h\in \fA, \norm{h}_{BV} \leq 1} \right\} \Big)^{\frac{1}{k}} &\leq  \Lambda_{BV}^{-1}, \\
            \limsup_{k\to\infty} \Big( \sup \left\{ \abs{\ell_{k}(h)} : {h\in \fA, \norm{h}_{L^\infty} \leq 1} \right\} \Big)^{\frac{1}{k}} &\leq  \Lambda_{L^\infty}^{-1};
        \end{aligned}
    \end{equation}
    And also satisfy:
    \begin{itemize}
        \item \(\ell_{1}(h_{0}) = 1\) and, for all \(k\in\bN\),
        \(\ell_{k+1}(h_{0}) = 0\);
        \item For all \(k\in\bN\), \(h\in \fA\),
        \(\ell_{k+1}(\cL h) = \ell_{k}(h)\).
    \end{itemize}    
\end{proposition}
\noindent
Section~\ref{sec:lin-funcs} is devoted to the proof of the above.

\begin{proof}[Proof of Theorem~\ref{thm:main}]
    The assumptions of the theorem imply that Proposition~\ref{prop:funct} holds.
    Consider first the case in which \({\norm{\cdot}}_{BV} \le {\norm{\cdot}}_{\fB}\).
    Let \(\overline{\fA}\) denote the completion of \(\fA\) with respect to \(\norm{\cdot}_{BV}\).
    By assumption, \(\cC^{\infty}(\cM) \subset \mathfrak{B}\) and \(\cL\) extends to a continuous operator on \(\mathfrak{B}\), and so \(\fA \subset \fB\).    
    Every complete subspace of a Banach space is a Banach space and so \((\overline{\fA}, \norm{\cdot}_{BV})\) is a Banach space.
    Furthermore, each of the statements of Proposition~\ref{prop:funct} concerning \(\ell_k\) extend to \(\overline{\fA}\) by continuity.
    In particular the estimate \eqref{eq:two-bounds} implies \ref{it:B1} for \(\Lambda = \Lambda_{BV}\).
    This means that the assumptions of Proposition~\ref{prop:abs} are satisfied with respect to  \((\overline{\fA}, \norm{\cdot}_{BV})\) and so the essential spectral radius of \(\cL : \overline{\fA} \to \overline{\fA}\) is at least \(\Lambda_{BV}\).
    Finally, since \(\overline{\fA} \subset \fB\), we use the fact that the essential spectral radius of \(\cL: \overline{\fA} \rightarrow \overline{\fA}\) is not greater than the essential spectral radius of \(\cL: \fB \rightarrow \fB\) (see~\cite[\S4]{BCJ22}).
    Consequently we conclude that the essential spectral radius of \(\cL: \fB \rightarrow \fB\)  is at least \(\Lambda_{BV}\).

    Alternatively, in the case in which \({\norm{\cdot}}_{L^\infty} \le {\norm{\cdot}}_{\fB}\), we let \(\overline{\fA}\) denote the completion of \(\fA\) with respect to \(\norm{\cdot}_{L^\infty}\) and proceed identically to obtain the analogous conclusion.
\end{proof}

\begin{remark}[observables]
    Understanding the statistical properties of \(\cC^\infty\) observables typically suffices, via approximation arguments, to determine the behaviour of less regular space of observables.
    For example, one could consider \(\cC^r\), Hölder, or indeed any spaced defined by fixing some modulus of continuity.
    In such a case the general results can be obtained by arguing about how ``rapidly'' the observables can be approximated by \(\cC^\infty\) ones (see e.g.  \cite[Corollary 2.3]{AM16} and \cite[Corollary 1]{Dolgopyat98}). The reference metric is the crucial detail we are fixing when we consider \(\cC^\infty(\cM)\) rather than the specific regularity.
\end{remark}

\begin{remark}[disc of spectrum]
    The proof of Theorem~\ref{thm:main}, together with the arguments in \cite[\S 4]{BCJ22}, implies a slightly stronger statement, namely that the complement of the unbounded component of the resolvent set contains a disc of radius \(\Lambda_{BV}\) or \(\Lambda_{L^{\infty}}\).
\end{remark}

\begin{remark}[partitions and artificial discontinuities]
    Often a convenient way to work with systems like we have here is to assume, by definition, that there exists a finite partition \(\Omega\) of a full measure subset of \(\cM\) such that, for each \(\omega \in \Omega\), \(\cF\) restricted to \(\omega\) has no critical points, is injective and admits a \(\cC^r\) extension to \(\overline{\omega}\). 
    We know that \(\Gamma \subset \bigcup_{\omega \in \Omega}\partial \omega \subset M\) but this might easily be a proper subset.
    We would like to assume that \(\Omega\) denotes the \emph{minimal partition} of \(\cM\) with the aforementioned properties.
    Unfortunately, for some systems, there might not be a unique such partition, e.g., a system which is smooth apart from a short ``slit'' as introduced in Section~\ref{sec:local-example}.
    In other words, working with such a definition can force artificial discontinuities to be introduced into the setting, an issue which we wish to avoid.
\end{remark}

\begin{remark}[comparison of estimates]
    \label{rem:compare}
    The different bounds~\eqref{eq:Lambda} are more or less useful in different settings, as illustrated by the following:
    If \(\len{\gamma_k}\) is uniformly bounded, as is the case for the cocycle examples of Section~\ref{sec:cocycle}, then \(\Lambda_{BV} = \Lambda_{L^\infty,2}\);
    If \(\Jac{\gamma}{\cF^{k}} \equiv 1\), as could be arranged for a cocycle example of Section~\ref{sec:cocycle} by choosing the fibre map \(S(x,y)\) to be a circle rotation, then \(\Lambda_{BV} = \Lambda_{L^\infty,1}\); 
    On the other hand, if there is contraction along \(\gamma_k\), for example a suitable piecewise smooth Anosov endomorphism, then \(\Lambda_{BV}\) will be strictly smaller than \(\Lambda_{L^\infty}\).
    However this framework is unsuitable for studying general systems with a contracting direction since one would need to consider the connection with anisotropic Banach spaces which permit distribution-like behaviour in the contracting direction. Nevertheless, we expect that many ideas developed in this work, particularly defining linear functionals by integration, will be useful for tackling the hyperbolic case. 
\end{remark}

\begin{remark}[Quasi-Hölder spaces]
Other than the BV spaces \cite{Liverani13, Thomine11}, relevant examples of Banach spaces used to prove a spectral gap for the transfer operator associated to multidimensional expanding maps with singularities are the quasi-H\"older spaces introduced in \cite{Saussol00}. 
By \cite[Proposition 3.4]{Saussol00} these spaces are continuously injected into \(L^{\infty}\), therefore Theorem \ref{thm:main} applies in this case with the second lower bound.
\end{remark}

\section{Discussion of the proper discontinuity assumptions}
\label{sec:discuss}

Differently from the one-dimensional case, where failing to be Markov (see~\cite{BCJ22} for further details) was sufficient to prove the lower bound on the essential spectral radius, in the present two-dimensional setting the geometry of the discontinuities led us to impose some extra assumptions. These assumptions were used to construct the linear functionals in Section~\ref{sec:lin-funcs} and are satisfied for three rather general examples presented in Section \ref{sec:examples}.
Here we discuss the distance between the proper discontinuity assumption (Definition~\ref{def:proper}) and failing to be Markov.
Let \(\cM\), \(\cF\) and \(\Gamma\) be as specified in Section~\ref{sec:results} but without requiring the existence of a proper discontinuity.

In the following we see that a combination of certain assumptions of a proper discontinuity implies that the \(\gamma_k\) curves are essentially pairwise disjoint.

\begin{lemma}
    \label{lem:disjoint-gamma}
    Suppose that, with respect to a curve \(\gamma \subset \Gamma\), \ref{it:A2} and \ref{it:A3} are satisfied.
    Then \(\gamma_j \cap \gamma_k \approxeq \emptyset\) whenever \(j,k\in\bN\), \(j \neq k\).
\end{lemma}

\begin{proof}
    We first show, as a consequence of \ref{it:A3}, that for all \(j,k \in \bN\),
    \begin{equation}
        \label{eq:disjoint-gamma}
        \gamma_{j} \cap \gamma_{k} \approxeq \emptyset
        \quad \text{implies that} \quad
        \gamma_{j+1} \cap \gamma_{k+1} \approxeq \emptyset.
    \end{equation}
    According to \ref{it:A3}, \(\cF^{-1}\gamma_{j+1} \cap \bigG \approxeq \gamma_{j}\).
    Let \(\tilde{\gamma}_{j+1} \subset \gamma_{j+1}\) satisfy \(\tilde{\gamma}_{j+1} \approxeq \gamma_{j+1}\) and \(\cF^{-1}\tilde{\gamma}_{j+1} \cap \bigG \subset \gamma_{j}\).
    In order to prove the above~\eqref{eq:disjoint-gamma}, it suffices to show that \(\gamma_{k+1} \cap \tilde{\gamma}_{j+1} \subset \cF(\gamma_k \cap \gamma_j)\).
    Let \(y \in \gamma_{k+1} \cap \tilde{\gamma}_{j+1}\).
    Since \(\gamma_{k+1} = \cF \gamma_k\) there exists \(x\in \gamma_k\) such that \(y = \cF x\).
    However, as observed above, \(\cF^{-1}\tilde{\gamma}_{j+1} \cap \gamma_k \subset \gamma_{j}\).
    We know that \(x \in \cF^{-1}\tilde{\gamma}_{j+1} \cap \gamma_k\) and so this shows that \(x\in \gamma_j\).
    This completes the proof of the above statement~\eqref{eq:disjoint-gamma}.

    We now take advantage of~\ref{it:A2}.
    Without loss of generality, suppose that \(j < k\).
    Iterating what we just proved~\eqref{eq:disjoint-gamma}, we know that,
    \begin{equation*}
        \gamma_{j} \cap \gamma_{k} \approxeq \emptyset
        \quad \text{whenever} \quad
        \gamma_{1} \cap \gamma_{k-j+1} \approxeq \emptyset.
    \end{equation*}
    Observe that \(k-j+1 \geq 2\) and that \(\gamma_1 \subset \Gamma_1\) and so, by~\ref{it:A2}, \(\gamma_{1} \cap \gamma_{k-j+1} \approxeq \emptyset\).
    Consequently we have shown that \(\gamma_j \cap \gamma_k \approxeq \emptyset\).    
\end{proof}

We recall that the map $\cF$ is said to be Markov if it admits a Markov partition.
Such a partition \(\Omega\) is a finite set of subsets of \(\cM\) which are each open, connected, have boundary consisting of a finite union of finite length \(\cC^1\) curves and together form a partition of a full measure subset of \(\cM\).
Moreover $\cF$ is $\cC^r$ and injective on each $\omega \in \Omega$ and $\omega \cap \omega' \neq \emptyset$ implies $\cF(\omega) \supset \omega'$, for each $\omega,\omega'  \in \Omega$.

\begin{lemma}
    \label{lem:not-Markov}
    Suppose that, with respect to \(\gamma \subset \Gamma\), \ref{it:A0}, \ref{it:A2} and \ref{it:A3} are satisfied and \(\len{\gamma_k}\) is uniformly bounded from below.
    Then \(\cF\) is not Markov.
\end{lemma}

\begin{proof}
    Suppose, for sake of a contradiction, that \(\cF\) is Markov and let \(\Omega\) denote the associated partition. 
    For convenience let \(\partial \Omega = \bigcup_{\omega \in \Omega} \partial \omega\).
    We know that \(\gamma \subset \Gamma \subset \partial \Omega\) since \(\cF\) is \(\cC^r\) on each \(\omega \in \Omega\).
    The Markov partition has the property that \(\cF (\partial \Omega) \subset \partial \Omega\) and so \(\gamma_k \subset \partial \Omega\) for all \(k\in \bN\).
    By Lemma~\ref{lem:disjoint-gamma} we know that \(\gamma_j \cap \gamma_k \approxeq \emptyset\) whenever \(j \neq k\).
    Since we assumed a uniform lower length for the \(\gamma_k\) we obtain a contradiction.
\end{proof}

Observe that assuming \(\cF\) to be expanding might not guarantee that \(\len{\gamma_k}\) is uniformly bounded from below since \({\left.\cF\right|_{\gamma_k}}\) need not be injective and indeed won't be in many basic examples (see Section~\ref{sec:cocycle} for such examples and \cite{Tsujii00} for an illustration of the singular possibilities of piecewise expanding maps).
Without modifying the argument, the assumption of the lemma can be weakened to requiring that \(\sum_{k} \len{\gamma_k}\) is a divergent series.
However this is still not sufficient to deal with all cases.

\subsection*{Assumption \ref{it:A0}}

In general, if the map has no discontinuity (or is Markov) and is \(\cC^r\) then the essential spectral radius can be made smaller and smaller for bigger and bigger $r$ (see e.g. \cite{CI91,GL03}). Consequently some assumption like \ref{it:A0} is required to obtain a result such as Theorem \ref{thm:main}.

Assumption \ref{it:A0} could be weakened to permit the higher dimensional version of the over-lapping mechanism which is seen in ``W-maps''~\cite[\S6]{Keller82} and unimodal expanding maps~\cite{BS08}.
These exhibit behaviour typically seen with discontinuities whilst being continuous.\footnote{The reason is that, since these maps preserve orientation on one side and reverse orientation on the other side, the push-forward of a smooth density will have jump discontinuities.}
It appears that the results obtained in this work would continue to hold for such a weakened assumption.

\subsection*{Assumption \ref{it:A1}}

Suppose that \(\gamma \subset \Gamma\) is a curve and that \(\cF : \gamma_{k-1} \to \gamma_{k}\) is injective, with the exception of a countable set of points, for every \(k\in\bN\). 
Assumption \ref{it:A1} can easily be satisfied by choosing \(\alpha_0 \equiv 1\) and then defining the other \(\alpha_k\) inductively.
On the other hand, when \(\cF : \gamma_{k-1} \to \gamma_{k}\) is many-to-one, then we require the correct relationship between \(\varphi\) and \(\Jac{\gamma_k}{\cF}\) at preimages else it would be impossible to satisfy the required relationship~\eqref{eq:good-weights}.
It is unreasonable to imagine that such a problem won't exist for some  map which fails to be Markov and yet satisfies all other properties.

\subsection*{Assumptions \ref{it:A2}, \ref{it:A3}}

Recall that \ref{it:A2} specified that
\(\gamma_k \cap \Gamma_1 \approxeq \emptyset \) for all \(k\geq 2\) and that \ref{it:A3} specified that \(\cF^{-1}\gamma_{k+1} \cap \bigG\approxeq \gamma_{k}\) for all \(k\in\bN\).
Lemma~\ref{lem:not-Markov} tells us that this, together with \ref{it:A0} and a lower bound on \(\len{\gamma_k}\), these assumptions suffice to imply that \(\cF\) is not Markov.
These assumptions are not as tight as desired, as illustrated with the following example.
This is the same possibility that already exists in the one-dimensional case \cite{BCJ22}.
E.g., suppose that \(f : [0,1] \to [0,1]\) is an expanding interval map and that \(f\) is discontinuous at two distinct points \(a,b \in [0,1]\) which satisfy \(f(a^{+}) = f(b^{+})\) but that \(\{ f^n (a^{+})\}_{n\in\bN} \) is an infinite set of points. 
In such a way the required assumption holds for the second iterate, although not from the very first. 
A two-dimensional example can then be built from this as a product or cocycle (as per Section~\ref{sec:cocycle}).

\section{Essential spectrum: an abstract result}
\label{sec:ess-spec}

In this section we prove Proposition~\ref{prop:abs}. 
Suppose that \(\lambda \in \bC\), \(\abs{\lambda} < \Lambda\) and define, 
\begin{equation*}
    \Xi_{\lambda} = \sum_{k=1}^{\infty}\lambda^k \ell_{k}.    
\end{equation*}
Take \(\underline{\Lambda} \in (|\lambda|, \Lambda)\). Assumption~\ref{it:B1} implies that there exists \(C > 0\) such that \(\abs{\ell_k(h)} \leq C \underline{\Lambda}^{-k} \norm{h}\) for all \(k\in \bN\), \(h\in \fB\).
Consequently
\[
    \abs{\Xi_{\lambda} (h)} \leq C \sum_{k=1}^{\infty} \abs{\lambda}^k   \underline\Lambda^{-k} \norm{h}
\]
and so \(\Xi_{\lambda}\) defines a linear functional on \(\fB\).
By \ref{it:B2}, \(\Xi_{\lambda} (h_{0}) =  \lambda \ell_{1}(h_{0}) = \lambda \) and so \(\Xi_{\lambda}\) is non-zero.
We also define the rank-one operator \(\cK : \fB \to \fB\),
\begin{equation}%
    \label{eq:def-K}
    \cK h = \ell_{1}(L h) \, h_{0}.
\end{equation}
Using assumption~\ref{it:B3} we calculate that,
\[
    \begin{aligned}
        L^{*} \Xi_{\lambda}
         & = \sum_{k=1}^{\infty}\lambda^k L^{*} \ell_{k}  = \lambda L^{*} \ell_{1} + \sum_{k=2}^{\infty}\lambda^k \ell_{k-1}             \\       
         & = \lambda L^{*} \ell_{1} + \lambda \sum_{k=1}^{\infty}\lambda^k \ell_{k} = \lambda L^{*} \ell_{1} + \lambda \ \Xi_{\lambda}.     
    \end{aligned}
\]
Since \(\Xi_{\lambda} (h_{0}) = \lambda\), we know that, for all \(h\in \fB\), \(\Xi_{\lambda} (\cK h) = \lambda \ell_{1}(L h)\).
I.e., \(\cK^{*} \Xi_{\lambda} = \lambda L^{*} \ell_{1}\).
Consequently, we have shown that,
\[
    (L - \cK)^*\Xi_{\lambda} = \lambda \ \Xi_{\lambda},
\]
whenever \(\abs{\lambda} < \Lambda\).
This means that the essential spectral radius of \(({  L}-\cK) : \fB \to \fB\) is at least\footnote{Indeed, the spectrum of $L-\cK$ equals that of $(L-\cK)^*$. Therefore, since the set $\{\lambda \in \bC : |\lambda|<\Lambda\}$ is not discrete, it is contained in the essential spectrum of $L-\cK$.} \(\Lambda\) and, since compact perturbations don't change the essential spectral radius, the same holds for \(L\). 
\qed

\section{Linear functionals}
\label{sec:lin-funcs}

In this section we construct a sequence of linear functionals and hence prove Proposition~\ref{prop:funct}.
Let \(\cF:\cM\to \cM\), \(\varphi : \cM \to \bC\) be fixed for the remainder of this section, satisfying the assumptions of the proposition. 
In particular this means that \(\cF\) has a proper discontinuity \(\gamma\) such that \ref{it:A0}, \ref{it:A1}, \ref{it:A2} and \ref{it:A3} are each satisfied with respect to the curves \(\{\gamma_k\}\) and associated \(\{\alpha_k\}\).
The proof will consist of defining the linear functionals and then proving that the required properties are satisfied.

For each \(h \in \fA\), let us define the ``jump'' of \(h\) at a point \(x\in \cM\) in the direction \(v \in \cT_x \cM\) as,%
\footnote{Here, \(\exp_x(v)\) denotes the exponential map based at \(x\), i.e., the application that sends \(v\) to \(\eta(1)\), where \(\eta\) is the geodesic with \(\eta'(x) = v\).}%
\[
    \Jmp(h, x, v) 
    = \lim_{\epsilon \rightarrow 0}
    \left[h\left(\exp_{x}(\epsilon v)\right) - h\left(\exp_{x}(-\epsilon v)\right) \right].
\]
Recall that \(\gamma \subset \Gamma \subset M\) is part of the boundary of the manifold \(M\).
For \(x \in \gamma\) let \(v(x) \in \cT_x M\) denote the unit vector which is orthogonal to \(\gamma\) and such that \(\exp_x(\epsilon v)\) is in \(M\) for small \(\epsilon>0\).
For each \(k\in \bN\), \(y \in \gamma_k\) let \(v_k(y) \in \cT_y \cM\) be defined as \(v_k(y) = D\cF^k(x) \ v(x)\) where \(y = \cF^k x\).
For each \(k\in \bN\), \(h\in \fA\), we define\footnote{If \(h: \gamma \to \bC\) and \(\gamma \subset \cM\) is a \(\cC^1\) curve, we systematically write \(\int_{\gamma} A(x) \ dx\) as shorthand for \(\smash{\int_{0}^{\len{\gamma}}} A(x(t)) \ dt\) with the understanding that \(x : [0,\abs{\gamma}] \to \gamma\) is a parametric representation of the curve \(\gamma\) which respects arc length.} the linear functional \(\ell_k : \fA \rightarrow \bR\),
\begin{equation}
    \label{eq:def-lin-func}
    \ell_{k}(h) = \int_{\gamma_{k}} \alpha_{k}(x)\Jmp(h, x, v_k(x)) \ dx. 
\end{equation}
In the remainder of this section we prove the required properties of these linear functionals.

\begin{proof}[{Proof of Proposition~\ref{prop:funct}, part 1}]
 	 Before dividing into cases, we will prepare a couple of estimates which follow directly from  \ref{it:A1}.
  
   Let  \(x \in \gamma_1\), \(k\in\bN\) and \(y = \cF^k x \in \gamma_{k+1}\).
    Iterating the assumed relationship \ref{it:A1} and recalling that  \(\alpha_1 \equiv 1\) implies that 
    \begin{equation}
        \label{eq:iterate-A1}
        \alpha_{k+1}(y)
        = \varphi_k(x)^{-1} \Jac{\gamma_1}{\cF^k}(x)^{-1}.
    \end{equation}
    It may happen that a given \(y\in \gamma_{k+1}\) has more than one preimage in \(\gamma_1\).
    However, the assumed relationship \ref{it:A1} means that we are working in the case where the term \(\varphi_k(\cdot) \Jac{\gamma_1}{\cF^k}(\cdot)\) is equal at each such preimage, as implied by this calculation.
    The above \eqref{eq:iterate-A1} implies that,
    \begin{equation}  
        \label{eq:alpha-est-A}
        \sup_{y\in \gamma_{k+1}} \abs{\alpha_{k+1}}(y)
        \leq \left(\inf_{x \in \gamma_1} \left| \varphi_k(x) \cdot \Jac{\gamma_1}{\cF^k}(x) \right|\right)^{-1}.
    \end{equation}
    For convenience, let \(\tilde\gamma_1 \subset\gamma_1\) be such that \(\cF^k\) is injective from \(\tilde\gamma_1\) to \(\gamma_{k+1}\). 
    Using again the above formula \eqref{eq:iterate-A1},
    \[
        \begin{aligned}
            \int_{\gamma_{k+1}} \abs{\alpha_{k+1}(y)} \ dy  
             & =   \int_{\tilde\gamma_{1}} \abs{\alpha_{k+1}(\cF^k x)} \Jac{\gamma_1}{\cF^k}(x) \ dx            \\
             & \leq \left( \inf_{x \in \gamma_1} \abs{\varphi_k(x)} \right)^{-1}  \int_{\tilde\gamma_{1}} \ dx.
        \end{aligned}
    \]
    Consequently, observing that \(\len{\tilde\gamma_{1}} \leq \len{\gamma_1}\), we have shown that 
    \begin{equation}  
        \label{eq:alpha-est-B}
        \int_{\gamma_{k+1}} \abs{\alpha_{k+1}(y)} \ dy  
        \leq \left( \inf_{x \in \gamma_1} \abs{\varphi_k(x)} \right)^{-1} \len{\gamma_1}.
    \end{equation}
    Having prepared the above estimates, we now divide the argument into the three cases \eqref{eq:Lambda}.
    
    \noindent
    \textbf{Case 1:}
    Suppose that \({\|h\|}_{BV} \le {\|h\|}_{\fB}\),  \(\forall h \in \fA\).
    The control given by the \(BV\) norm implies that there exists%
    \footnote{See e.g., \cite[Definition 3.67, Theorem 3.86]{AFP00}. For this it is essential that the linear function is defined by integrating along a curve rather than considering jumps at points.}
    \(C>0\) such that, for each $k\in \bN$, \(h\in BV(\cM)\),
    \begin{equation}
        \label{eq:BV-est}
        \int_{\gamma_k} \abs{\Jmp(h,x,v_k(x))} \ dx \le C \norm{h}_{BV}\le C \norm{h}_{\fB}.  
    \end{equation}
    In particular, the bound does not depend on the length of \(\gamma_k\).
    Since,
    \[
        \begin{aligned}
            \abs{\ell_{k+1}(h)} & = \abs{\int_{\gamma_{k+1}} \alpha_{k+1}(x)\Jmp(h, x, v_{k+1}(x)) \ dx}                            \\
                                & \leq \sup_{y\in \gamma_{k+1}} \abs{\alpha_{k+1}}(y)  \int_{\gamma_k} \abs{\Jmp(h,x,v_k(x))} \ dx,
        \end{aligned}
    \]
    the above estimates \eqref{eq:BV-est}, \eqref{eq:alpha-est-A} imply that 
    \[
        \abs{\ell_{k+1}(h)}
        \leq C \left(\inf_{x \in \gamma_1} \left| \varphi_k(x) \cdot \Jac{\gamma_1}{\cF^k}(x) \right|\right)^{-1}  \norm{h}_{\fB}.
    \]
    
    \noindent
    \textbf{Case 2:}
    Suppose that \({\|h\|}_{L^{\infty}} \le {\|h\|}_{\mathfrak{B}}\), \(\forall h \in \fA\).
    We take advantage of the trivial estimate,
    \begin{equation}
        \label{eq:trivial-est}
        \left|\Jmp(h,x,v_k(x)) \right| \le 2 {\|h\|}_{L^{\infty}} \leq 2  \norm{h}_{\fB}. 
    \end{equation}
    Since
    \[
        \begin{aligned}
            \abs{\ell_{k+1}(h)} & = \abs{\int_{\gamma_{k+1}} \alpha_{k+1}(x)\Jmp(h, x, v_{k+1}(x)) \ dx}                                                 \\
                                & \leq  \int_{\gamma_{k+1}} \abs{\alpha_{k+1}(y)} \ dy  \ \ \sup_{x\in \gamma_{k+1}} \left|\Jmp(h,x,v_{k+1}(x)) \right|,
        \end{aligned}
    \]
    we use the above \eqref{eq:alpha-est-B} and obtain the estimate,
    \[
        \abs{\ell_{k+1}(h)}
        \leq  2 \len{\gamma_1} \left( \inf_{x \in \gamma_1} \abs{\varphi_k(x)} \right)^{-1} \norm{h}_{\fB}.
    \]
    This corresponds to \(\Lambda_{L^{\infty},1}(\cF,\varphi, \gamma)\) in \eqref{eq:Lambda}.
    
    \noindent
    \textbf{Case 3:}
    Again suppose that \({\|h\|}_{L^{\infty}} \le {\|h\|}_{\mathfrak{B}}\), \(\forall h \in \fA\).
    Estimating as above~\eqref{eq:trivial-est},
    \[
        \begin{aligned}
            \abs{\ell_{k+1}(h)} & = \abs{\int_{\gamma_{k+1}} \alpha_{k+1}(x)\Jmp(h, x, v_{k+1}(x)) \ dx}             \\
                                & \leq 2 \int_{\gamma_{k+1}} \abs{\alpha_{k+1}(y)} \ dy \ {\|h\|}_{\mathfrak{B}}.
        \end{aligned}
    \]
    Using also the prior estimate \eqref{eq:alpha-est-A},
    \[
        \abs{\ell_{k+1}(h)} \leq 2
        \left(\inf_{x \in \gamma_1} \left| \varphi_k(x) \cdot \Jac{\gamma_1}{\cF^k}(x) \right|\right)^{-1} \ \len{\gamma_{k+1}}  {\|h\|}_{\mathfrak{B}}.
    \]
    This corresponds to \(\Lambda_{L^{\infty},2}(\cF,\varphi, \gamma)\) in \eqref{eq:Lambda}.
\end{proof}

Before proving the next claim of Proposition~\ref{prop:funct} we will introduce a few further lemmas.

\begin{lemma}
    \label{lem:preimages}
    The cardinality of \(\cF^{-1}y\) is bounded, uniformly for \(y\in\cM\).
\end{lemma}

\begin{proof}
    Let \(\{U_j\}_j\) be the connected components of \(\cM \setminus \Gamma\).
    This is a finite set since \(\Gamma\) is the union of a finite number of smooth curves of a finite length.
    Let \(\delta>0\).
    For each \(j\), let \(\{V_{j,k}\}_k\) be a finite cover of \(U_j\) comprising of subsets of \(U_j\) of diameter less than \(\delta\).
    Since the derivative of \(\cF\) is bounded we may choose \(\delta\) sufficiently small such that \(\left.\cF\right|_{V_{j,k}}\) is 1-to-1.
    Consequently the cardinality of \(\cF^{-1}y\) is bounded by the cardinality of \(\{V_{j,k}\}_{j,k}\).
\end{proof}

Let \(\sigma \subset \Gamma \subset \cM\) denote the finite set of points of intersection of the curves in \(\Gamma\).
Consequently, for any point in \(\Gamma \setminus \sigma\), there exists a neighbourhood such that the intersection of this neighbourhood with \(\Gamma\) is a curve.

\begin{lemma}
    \label{lem:first-image}
    Let \(y \in \cM\) and set\footnote{The finite set of preimages of \(y\), divided into those in \(\partial M\) and those in the interior.} \(\cA_1(y) = \partial M \cap F^{-1} y\) and \(\cA_2(y) = F^{-1} y \setminus \cA_1(y) \).
    There exists \(U\), a neighbourhood of \(y\),
    and, for \(x \in \cF^{-1} y\), there exist mutually disjoint neighbourhoods\footnote{Although \(V_x \subset M\), writing \({\left.\cF\right|}_{V_x}\) we consider \(V_x\) as a subset of \(\cM\), discarding \(\partial M\).} \(V_x\) of \(x\), \(\cF : V_x \to U\) injective,
    such that, for all \(h\in \fA\), on \(U \setminus \Gamma_1\),
    \begin{equation}
        \label{eq:precise}
        \cL h = 
        \sum_{x \in \cA_2(y)} \left(\varphi \cdot h \right)\circ {\left.\cF\right|}_{V_x}^{-1}
        + \sum_{x \in \cA_1(y)} \left(\varphi \cdot h \right)\circ {\left.\cF\right|}_{V_x}^{-1}
        \cdot  \mathbf{1}_{\cF V_x}.
    \end{equation}
    Moreover, if \(y \in \gamma_1 \setminus F (\iota^{-1} \sigma)\), there exists \(x_0 \in \cA_1(y)\) such that \(\tilde{x}_0 \notin \cA_1(y)\) where \(\tilde{x}_0 \in \partial M\) denotes\footnote{I.e., \(\tilde{x}_0 \in M\) denotes the point that is equal to \(x_0\) as points in \(\cM\) yet distinct in \(M\).} the point such that \(x_0 \neq \tilde{x}_0\) yet \(\iota(x_0) = \iota(\tilde{x}_0)\).
\end{lemma}

\begin{proof}
    The form of the transfer operator follows from the definition \eqref{eq:def-L}, the only thing to note is that the term \( \mathbf{1}_{\cF V_x}\) can be removed from the first sum.
    This is because \(x \in \cA_2(y)\) means that \(\cF V_x\) is a neighbourhood of \(y\) (as a point in \(\cM\)). 
    Consequently, since \(\cA_2(y)\) is a finite set, \(U\) can be chosen so that \(\cF V_x\) covers \(U\) for all \(x \in \cA_2(y)\).
    
    The final statement follows from assumption \ref{it:A0}.
    Suppose this statement were false, then, for any \(x_0 \in \cA_1(y)\), there exists \(\tilde{x}_0 \in \cA_1(y)\) such that \(x_0\) and \(\tilde{x}_0\) are equal as points in \(\cM\).
    Let \(x=\iota (x_0) = \iota (\tilde{x}_0) \in \cM\) and let \(W_{y}\subset\cM\) be a neighbourhood of \(y=\cF x\).
    Notice that any neighbourhood of \(x\) is covered by \(\iota(W_{x_0} \cup W_{\tilde{x}_0})\) where \(W_{x_0} \subset M\), \(W_{\tilde{x}_0} \subset M\) are neighbourhoods of \(x_0\), \(\tilde{x}_0\) respectively. 
    Shrinking the neighbourhoods we can guarantee that \(\cF(\iota(W_{x_0} \cup W_{\tilde{x}_0})) \subset W_{y}\).
    Consequently \(\cF\) is continuous at \(x\) in contradiction of \ref{it:A0}.
\end{proof}

The following result concerns the structure of elements of \(\fA\).

\begin{lemma}
    \label{lem:struct-of-functions}
    \(\cL^n g\) is uniformly \(\cC^1\) on \(\cM \setminus \bigcup_{p=1}^{n} \Gamma_p\) for all \(n \in \bN_{0}\), \(g \in \cC^{\infty}(\cM)\).
\end{lemma}

\begin{proof}
    The case \(n=0\) is immediate.
    Let us now suppose the statement is true for \(n\) and consider the \(n+1\) case. 
    Let \(h=\cL^n g\) and so \(\cL^{n+1}g = \cL h\).
    By the inductive assumption, we know that \(h\) is uniformly \(\cC^1\) on \(\cM \setminus \bigcup_{p=1}^{n} \Gamma_p\).
    Consider \(y \notin \bigcup_{p=1}^{n+1} \Gamma_p\).
    To complete the proof it suffices to show that \(\cL h\) is uniformly \(\cC^1\) in a neighbourhood of any such \(y\).
    Using Lemma~\ref{lem:first-image}, there exists \(U\), a neighbourhood of \(y\) and, for each \(x\in \cF^{-1}y\), a neighbourhood \(V_x\) such that, on \(U\),
    \[
        \cL h = 
        \sum_{x\in \cF^{-1}y} \left(\varphi \cdot h \right)\circ {\left.\cF\right|}_{V_x}^{-1}
    \]
    where \(\cF : V_x \to U\) is bijective for each \(x\).
    We can consider this reduced version of the formula~\eqref{eq:precise} because of the following.
    Since \(y \notin \Gamma_1\) the second sum from the lemma, the one corresponding to \(\cA_1(y)\), can be omitted and we can assume that \(U \subset \cM \setminus \Gamma_1\).
    Moreover, shrinking the neighbourhoods as required, we may choose \(V_x \subset \cM \setminus \bigcup_{p=1}^{n} \Gamma_p\) since \(y \notin \cF (\bigcup_{p=1}^{n} \Gamma_p) = \bigcup_{p=2}^{n+1} \Gamma_p\).
    As such, \(\varphi \cdot h \) is uniformly \(\cC^1\) on each \(V_x\).
    Finally observe that \(\cF^{-1} : U \to V_x\) is uniformly \(\cC^1\) since \(\cF\) is \(\cC^r\) and uniformly non-singular.
\end{proof}

\begin{remark}
    Since \(M\) is a manifold with boundary, the boundary is locally homeomorphic to the half-space \(\{(x,y)\in\bR^2 : y \geq 0\}\).
    We also know that \(F : M \to \cM\) is \(\cC^r\).
    If \(V \subset M\) is a small neighbourhood of an interior point \(x \in M\) then \(F V \subset \cM\) will be a neighbourhood of \(F x\).
    On the other hand, if  \(W \subset M\) is a small neighbourhood of a boundary point \(x \in M\) then \(F W \subset \cM\) won't be a neighbourhood of \(F x\), it will only be ``half''.
\end{remark}

\begin{proof}[{Proof of Proposition~\ref{prop:funct}, part 2}]
    We will construct \(h_{0} \in \fA \) such that 
    \begin{equation}
        \label{eq:begin}
        \ell_{1}(h_{0}) = 1
        \quad \text{and} \quad 
        \ell_k(h_{0}) = 0 \text{ for all \(k \geq 2\)}.
    \end{equation}
    Let \(h \in \cC^{\infty}(\cM)\), to be chosen shortly, and let \(h_0 = \cL h\). 
    By Lemma~\ref{lem:struct-of-functions}, \(\cL h\) is \(\cC^1\) on  \(\cM \setminus \Gamma_1\).
    Whenever \(k\geq 2\), assumption \ref{it:A2} tells that \(\gamma_k \cap \Gamma_1 \approxeq \emptyset\).
    Consequently, \(\Jmp(\cL h, x, v_{k}(x))  = 0\) for all \(x\) in a full measure subset of \(\gamma_k\).
    We have thereby shown that \(\ell_k(\cL h) = 0\) whenever \(k \geq 2\).
    It remains to fix \(h_0\) and show that \(\ell_1(h_0) = 1\).

    Fix \(y \in \gamma_1 \setminus F (\iota^{-1} \sigma)\).
    Lemma~\ref{lem:first-image} tells us that, in \(U\setminus \Gamma_1\), where \(U\) is a neighbourhood of a point \(y_0 \in \gamma_1\),
    \[
        \cL h = 
        \sum_{x \in \cA_2} \left(\varphi \cdot h \right)\circ {\left.\cF\right|}_{V_x}^{-1}
        + \sum_{x \in \cA_1} \left(\varphi \cdot h \right)\circ {\left.\cF\right|}_{V_x}^{-1}
        \cdot  \mathbf{1}_{\cF V_x}.
    \]
    The first sum of terms contributes nothing to \(\Jmp(\cL h, y, v_{1}(y))\) because they are all continuous in \(U\) (recall that \(h \in \cC^{\infty}(\mathcal{M})\)). 
    Let \(x_0 \in \cA_1 \subset M\), \(x_0 \in \gamma\), be as given by Lemma~\ref{lem:first-image}.
    We ensure that \(y_0\) was chosen such that the intersection of \(V_{x_0}\) with \(\Gamma\) is a simple curve, a subset of \(\iota \gamma\). 
    We choose \(h \in \cC^{\infty}(\cM)\) such that \(\varphi \cdot h\) is positive in \(V_{x_0}\), the neighbourhood of \(\iota x_0 \in \cM\), and zero elsewhere. 
    We may assume that this neighbourhood is sufficiently small that it does not intersect any other curves in \(\Gamma\) except for \(\gamma\). This means that \({\Jmp(\cL h, \cF x_0, v_{1}(\cF x_0))}  = {\varphi(x_0)h(x_0)}\) (or \(-\varphi(x_0)h(x_0)\)). 
    Consequently \( \Jmp(\cL h, y, v_{1}(y))\) is positive (or negative, but say positive) for a positive measure subset of \(y \in \gamma_1\), given by \( \gamma_1 \cap \cF V_{x_0}\). 
    Furthermore, we can even ensure that \( \Jmp(\cL h, y, v_{1}(y))\) is everywhere non-negative  in \(\gamma_1\) since, by eventually making \(\varphi \cdot h\) positive in even a smaller neighbourhood of \(x_0\) and zero elsewhere, we can consider points in \(\gamma_1\) coming from just one side of \(\Gamma\) (so that \({\Jmp(\cL h, y, v_{1}(y))}\) does not change sign varying \(y \in \gamma_1\)). 
    This in turn implies that \(\ell_{1}(\cL h)\) is non-zero.
    Scaling appropriately, we insure that \(\ell_{1}(h_0) = 1\).    
\end{proof}

It remains to prove the final claim of Proposition~\ref{prop:funct} which, together with Lemma~\ref{lem:first-image}, makes use of the following two results.

\begin{lemma}
    \label{lem:jump-image}
    Let \(h \in \fA\). For any \(k \in \bN\), there exists \(\widetilde{\gamma}_k\) with  \(\widetilde{\gamma}_k \approxeq \gamma_k\) such that for all \(x \in \widetilde{\gamma}_k\) and any vector \(v \in \cT_{x}\cM\), \(v\) not parallel to \(\gamma_{k}'(x)\) and \(v \neq 0\), we have
    \begin{equation}
        \label{eq:sub-lemma}
        \lim_{\epsilon \to 0^+} h\circ \cF \left(\exp_{x}(\epsilon v)\right)
        = \lim_{\epsilon \to 0^+} h\left(\exp_{\cF x}(\epsilon \overline{v})\right),
    \end{equation}
    where \(\overline{v} = D_x\cF v \in \cT_{\cF x}\cM\).
\end{lemma}

\begin{proof}
    Suppose that \(h = \cL^n g\) for some \(g\in \cC^\infty(\cM)\).
    Consider the two paths,
    \[
        \rho(\epsilon) =  \cF \left(\exp_{x}(\epsilon v)\right),
        \quad \quad
        \tilde{\rho}(\epsilon) = \exp_{\cF x}(\epsilon \overline{v}).
    \]
    We have to show that \(\lim_{\epsilon \rightarrow 0}h \circ \rho (\epsilon)\) is equal to \(\lim_{\epsilon \rightarrow 0}h \circ \tilde{\rho} (\epsilon)\).
    Observe that \(\rho(0) = \tilde{\rho}(0) = \cF x\) and the two paths are tangent at \(0\). 
    We still need to work a little because \(h\) might fail to be continuous at any point in \(\bigcup_{j=1}^n\Gamma_j\).
    For the purpose of this proof, we say that a point \(y\in \gamma_{k+1}\) is a \emph{regular} point if there exists \(\eta>0\) such that the \(\eta\)-size neighbourhood of \(y\), \(U_\eta(y)\subset \cM\) satisfies \(U_\eta(y) \cap \bigcup_{j=1}^n\Gamma_j = \gamma_{k+1}\).
    Since \(\bigcup_{j=1}^n\Gamma_j\) is composed of a finite number of finite length curves, there can be at most a finite number of points which fail to be regular.
    Let \(\widetilde{\gamma}_k \subset \gamma_k\) be the set of all  points \(x\) such that \(y = \cF x \in \gamma_{k+1}\) is regular.
    By the above considerations, \(\widetilde{\gamma}_k \approxeq \gamma_k\).
    Observe that \(U_{\eta}(\cF x) \setminus \gamma_{k+1}\) consists of two connected components, the boundary between the two given by the curve \(\gamma_{k+1}\).
    By Lemma~\ref{lem:struct-of-functions}, \(h\) is continuous in each of these two components.     
    Now, the fact that \(v \in \cT_{x}\cM\) is not parallel to \(\gamma_k'(x)\), implies that \(\overline{v}\) is not parallel to \(\gamma_{k+1}'(\cF x)\) and the ranges of \(\rho(\epsilon)\) and, for all \(\epsilon\) small enough, \(\widetilde \rho (\epsilon)\) belong to the same connected component of \(U_{\eta}(\cF x)\) on which \(h\) is continuous. 
    Consequently the two limits coincide.
\end{proof}

\begin{lemma}
    \label{lem:shift-heart}
    For all \(k\in\bN\) there exists \(\widetilde{\gamma}_{k+1} \approxeq \gamma_{k+1} \) such that for all \(y \in \widetilde{\gamma}_{k+1} \), \(h \in \fA\),
    \[
        \Jmp(\cL h, y, v_{k+1}(y)) 
        = \sum_{\substack{x \in \gamma_k, \\ \cF x = y}}\varphi (x) \Jmp(h,x,v_{k}(x)).
    \]
\end{lemma}

\begin{proof}
    Let \(h = \cL^n g\) for some \(g\in \cC^\infty(\cM)\).
    By Lemma~\ref{lem:struct-of-functions}, \(h\) is uniformly \(\cC^1\) on \(\cM \setminus \bigcup_{p=1}^{n} \Gamma_p\).
    By \ref{it:A3},
    \(\cF^{-1}\gamma_{k+1} \cap \bigG \approxeq \gamma_{k}\).
    Additionally, by \ref{it:A2},
    \(\gamma_{k+1} \cap \Gamma_1 \approxeq \emptyset\).
    Consequently, let \(\widetilde{\gamma}_{k+1} \approxeq \gamma_{k+1}\) denote the subset of \(\gamma_{k+1}\) such \(\cF^{-1}{\widetilde\gamma}_{k+1} \cap \bigG \subset \gamma_{k}\) and \(\widetilde\gamma_{k+1} \cap \Gamma_1 = \emptyset\).
    
    Applying Lemma~\ref{lem:first-image}, for a point \(y\in \widetilde{\gamma}_{k+1} \subset \gamma_{k+1}\), we obtain neighbourhoods which, following the notation of the lemma, are denoted \(U\), \(\{V_x\}\), such that 
    \(\Jmp(\cL h, y, v_{k+1}(y))\) is equal to the limit as \(\epsilon \to 0\) of
    \begin{multline}
        \label{eq:jump-func}
        \sum_{x \in \cF^{-1}y}
        \Bigl[ \left(\varphi \cdot h\right)
        \circ {\left.\cF\right|}_{V_x}^{-1}\left(\exp_{y}(\epsilon v_{k+1}(y))\right) 
        \cdot \mathbf{1}_{\cF V_x}\left(\exp_{y}(\epsilon v_{k+1}(y))\right)\\
        - \left(\varphi \cdot h\right)
        \circ {\left.\cF\right|}_{V_x}^{-1}\left(\exp_{y}(-\epsilon v_{k+1}(y))\right) 
        \cdot \mathbf{1}_{\cF V_x}\left(\exp_{y}(-\epsilon v_{k+1}(y))\right)\Bigr].
    \end{multline}
    Recall the notation from before (Lemma~\ref{lem:first-image}), that 
    \(\cA_1 = \partial M \cap \cF^{-1} y\) and \(\cA_2 = \cF^{-1} y \setminus \cA_1 \).
    As mentioned above, by \ref{it:A2} we know that
    \(\widetilde\gamma_{k+1} \cap \Gamma_1 \approxeq \emptyset\) 
    and so, in this case, \(\cA_1\) is an empty set and we need only to consider \(\cA_2\).
    This means that, for all \(x\in \cF^{-1}y\), \(\cF V_x\) is a neighbourhood of \(y\) so the indicator functions in the above sum can be discarded. 

    Since \(\cF^{-1}{\widetilde\gamma}_{k+1} \cap \bigG \subset \gamma_{k}\), if \(x\in \cF^{-1}\widetilde{\gamma}_{k+1}\) then \(x\in\gamma_k\) or \(x \notin \widehat{\Gamma}\). 
    In the above~\eqref{eq:jump-func} all the summands are zero, except the ones corresponding to \(x \in \gamma_k\).
    This is because, if \(x \notin \widehat{\Gamma}\), then \(\varphi \cdot h\) is \(\cC^1\) in a neighbourhood of \(x\) and so the positive and negative part of the summand cancel out.
    Here we have used the fact that  \(h\) is \(\cC^1\) on \(\cM \setminus \bigcup_{p=1}^{n} \Gamma_p\) and \(x \notin \bigcup_{p=1}^{n} \Gamma_p\).
    
    Taking advantage of Lemma~\ref{lem:jump-image}, \( \Jmp(\cL h, y, v_{k+1}(y))\) is equal to, for all \(y \in \widetilde{\gamma}_{k+1}\), 
    \begin{equation}
        \label{eq:one-term}
        \sum_{\substack{x \in \gamma_k \\ \cF x = y}} \lim_{\epsilon \rightarrow 0^+}
        \left[  (\varphi \cdot h) \left(\exp_x(\epsilon v_{k}(x))\right)
            - (\varphi \cdot h) \left(\exp_x(-\epsilon v_{k}(x))\right)\right].
    \end{equation}
    We conclude since \(\varphi\) is continuous in \(x\) (with the exception of at most a finite set of points, which we can always remove by redefining \(\widetilde{\gamma}_{k+1}\)). 
    Indeed, since
    \(\tilde{\gamma}_{k+1} \cap \Gamma_1 = \emptyset\), for all \(y \in \widetilde{\gamma}_{k+1}\), the above~\eqref{eq:one-term} is equal to
    \[
        \begin{aligned}
             & \sum_{\substack{x \in \gamma_k                     \\ \cF x = y}}\varphi(x)  \lim_{\epsilon \rightarrow 0^+} \biggl ( h \left(\exp_x(\epsilon v_{k}(x))\right)  
            -  h\left(\exp_x(-\epsilon v_{k}(x))\right) \biggr )  \\
             & \quad \quad \quad \quad \quad = \sum_{\substack{x \in \gamma_k \\ \cF x = y}}\varphi(x)  \Jmp(h, x, v_{k}(x))
        \end{aligned}
    \]
    as required by the statement of the lemma.
\end{proof}

\begin{proof}[{Proof of Proposition~\ref{prop:funct}, part 3}]
    Recalling the definition \eqref{eq:def-lin-func} of the linear functionals \(\ell_k\), Lemma~\ref{lem:shift-heart} implies that
    \[
        \begin{aligned}
            \ell_{k+1} (\cL h)
             & = \int_{\gamma_{k+1}} \alpha_{k+1}(y)\Jmp(\cL h, y, v_{k+1}(y)) \ dy      \\
             & = \int_{\gamma_{k+1}}  \alpha_{k+1}(y) \sum_{\substack{x \in \gamma_k, \\ \cF x = y}} \ \varphi(x) \ \Jmp(h, x, v_{k}(x)) \ dy.
        \end{aligned}
    \]
    Changing variables, the above is equal to   
    \[
        \int_{\gamma_k}\alpha_{k+1}(\cF x) \ \varphi(x) \ \Jmp(h, x, v_{k}(x))\ \Jac{\gamma_k}{\cF}(x) \ dx.
    \]
    The relationship \eqref{eq:good-weights} for the \(\alpha_k\) \ref{it:A1} means that, 
    \[
        \ell_{k+1}(\cL h) = \int_{\gamma_k}\alpha_{k}(x)  \ \Jmp(h, x, v_{k}(x)) \ dx = \ell_{k}(h),
    \]
    as required for the final claim of Proposition~\ref{prop:funct}.
\end{proof}

\section{Examples}
\label{sec:examples}

In this section we discuss some examples for which our results apply, namely we exhibit maps admitting a proper discontinuity for which Theorem~\ref{thm:main} holds.
In the first example, concerning two-dimensional piecewise affine expanding maps, we can say more: our lower bound for the essential spectral radius coincides with the upper bound known in the literature.

\subsection{Piecewise affine expanding maps}
\label{sec:affinemap}

Let \(\beta \in (1,\infty)\) be non-algebraic and \(\cM = \bT^2\). 
Let \(\cF_A: \cM \rightarrow \cM\) be defined through its action on the fundamental domain \([0,1)^2\),
\begin{equation}
    \label{eq:cF_A}
    \cF_A \left(x,
            y\right) = 
                \left( [\beta x + y],[ 2y]\right) ,
\end{equation}
where we write \([x] = x - \lfloor x \rfloor\) for the fractional part of any real number \(x\).
Let \(\Gamma = \{0\}\times \bT^1\). 
Let \(\varphi: \mathcal{M} \rightarrow \bC\) be as per Theorem~\ref{thm:main}, i.e., \(\cC^1\) on \(\cM \setminus \Gamma\), uniformly bounded and uniformly bounded away from zero.
Note that \(\cF_A\) is \(\cC^{\infty}\) on \(\cM \setminus \Gamma\), but for any \(y \in \bT^1\),
\begin{equation}
    \label{eq:not-equal}
    \lim_{x \rightarrow 1^{-}}\cF_A (x,y) \neq \lim_{x \rightarrow 0^{+}}\cF_A (x,y).    
\end{equation}
According to Section~\ref{sec:results}, we consider the extension of \(\cF_A\), denoted \(F_A:M \rightarrow \cM\), where \(M = [0,1] \times \bT^1\).
Furthermore, \(M\) is a smooth manifold whose boundary is \(\partial M = \gaml \cup \gamr\), where
\[
    \gaml = \left\{0\right\}\times \bT^1, \quad \gamr= \left\{1\right\} \times \bT^1.
\] 
Note that, doing so, we have ``doubled" the discontinuity set such that the extension is well defined. 
For all \(k \ge 1\) we set $\widetilde\beta_k=\frac 12\sum_{\ell=0}^{k-1}\big( \frac{\beta}{2}\big)^{\ell}$ and \(\widetilde \beta_0 = 0\). 
Observe that \(\{\widetilde\beta_k\}_{k}\) is a strictly increasing sequence and, in particular, that \(\widetilde\beta_k \neq \widetilde\beta_{k'}\) whenever \(k\neq k'\).
Next, let \(v = \left(\begin{smallmatrix}0\\1\end{smallmatrix}\right)\) be the unit tangent vector to the curve \(\gamr\) (or \( \gaml\)). 
One has that\footnote{Here we systematically write \(D\cF_A^{k}v\) to mean \(D\cF_A^{k-1} DF_A v\), \(\cF_A^k\gamr\) to mean \(\cF_A^{k-1}F_A\gamr\) and similar for \(\gaml\).}  
\[
    w_k = D\cF_A^{k}v = \begin{pmatrix}
        \beta^k &  2^{k}\tilde\beta_k \\
        0       &  2^k
    \end{pmatrix}\begin{pmatrix}0 \\
        1
    \end{pmatrix} = \begin{pmatrix}2^k \widetilde \beta_k \\
        2^k
    \end{pmatrix},
\]
is tangent to \(\cF_A^k\gamr\) (or \(\cF_A^k\gaml\)).  
This shows that both \(\cF_A^k \gamr\) and \(\cF_{A}^k\gaml\) are unions of straight lines with slopes \(\widetilde \beta^{-1}_k\). 
Note also that, for all \(x \in \cF^k_A\gamr\), 
\begin{equation}
    \label{eq:phi-det-A}
    \Jac{ \cF^k_A\gamr}{\cF}(x) = \frac{\|D\cF_A w_{k}\|}{\|w_{k}\|}.
\end{equation}

\begin{lemma}
    \label{lem:images}
    \mbox{}\\
    \noindent
    (i) For each $k \in \bN$ and \(a \in \cF_A^k\gamr\), there exist \( t\in [0,1],\{n_{j}\}_{j=0}^{k-1}, m \in \bZ\)  such that
    \[
        a = \left( \beta^k + \textstyle\sum_{j=0}^{k-1}\beta^j 2^{k-1-j}t + \sum_{j=0}^{k-1}\beta^j n_j, \ 2^k t + m \right).
    \]
    (ii) For each \(k  \in \bN\) and \(b \in \cF_A^k\gaml\), there exist \(\bar t \in [0,1]\), \(\{\bar n_{j}\}_{j=0}^{k-1}, \bar m \in \bZ\) such that
    \[
        b = \left( \textstyle \sum_{j=0}^{k-1}\beta^j 2^{k-1-j}\bar t + \sum_{j=0}^{k-1}\beta^j \bar n_j, \ 2^k \bar t + \bar m \right).
    \]
\end{lemma}

\begin{proof}
    (\textit{i}) If \(k = 1\), a direct computation yields
    \[
        \cF_A \gamr =  \bigl\{\left([\beta + t], [2t] \right)  \bigr\}_{t \in [0,1]},
    \]
    which is in agreement with the statement choosing $n_0 \in \bZ$, \(m \in \{-1,0,1\}\) depending on \(t\).
    Let us assume that it is true for \(k \ge 2\). Then any \(a \in \cF_A^{k+1} \gamr\) is of the form
    \[
        \begin{split}
            &\Biggl[\begin{pmatrix}
                \beta & 1 \\ 
                0     & 2
            \end{pmatrix}\begin{pmatrix}\beta^k + \sum_{j=0}^{k-1}\beta^j 2^{k-1-j}t + \sum_{j=0}^{k-1}\beta^j n_j \\
                2^k t + m
            \end{pmatrix}\Biggr] \\
            &= \begin{pmatrix}
                \left[ \beta^{k+1} + \sum_{j=0}^{k-1}\beta^{j+1} 2^{k-1-j}t + \sum_{j=0}^{k-1}\beta^{j+1} n_j + 2^k t + m\right] \\
                \left[2^{k+1}t + 2m\right]
            \end{pmatrix},
        \end{split}
    \]
    for some \(t \in [0,1]\). Therefore, there exist \(t \in [0,1]\), \(\widetilde n,  \widetilde m \in \bZ\) such that the coordinates of \(a\) are
    \[
        \begin{pmatrix}
            \beta^{k+1} + \sum_{j=0}^{k}\beta^{j} 2^{k-j}t + \sum_{j=1}^{k}\beta^{j} n_{j-1} + \widetilde{n} \\
            2^{k+1}t + \widetilde m
        \end{pmatrix}.
    \]
    Renaming the integers \(\widetilde n, \widetilde{m}\) and \(n_j\) properly, we have proved the statement for $k+1$ and conclude by induction.
    
    (\textit{ii}) If $k=1$,
    \[
        \cF_{A}\gaml = \bigl\{\left(\left[\bar t\right] , \left[2\bar t\right] \right)  \bigr\}_{\bar t \in [0,1]},
    \]
    which is in agreement with the statement choosing $\bar n_0,$ \(\bar m \in \{-1,0,1\}\) depending on \(t\). Assuming that is true for \(k \ge 2\), any $b\in \cF_A^{k+1} \gaml$ is of the form
    \[
        \begin{split}
            &\left[\begin{pmatrix}
                \beta & 1 \\ 
                0     & 2
            \end{pmatrix}\begin{pmatrix} \sum_{j=0}^{k-1}\beta^j 2^{k-1-j}\bar t + \sum_{j=0}^{k-1}\beta^j \bar n_j \\
                2^k \bar t + \bar m
            \end{pmatrix} \right]\\
            &= \begin{pmatrix}
                 & \left[\sum_{j=0}^{k-1}\beta^{j+1} 2^{k-1-j}\bar t + \sum_{j=0}^{k-1}\beta^{j+1} \bar n_j + 2^k \bar t +\bar  m\right] \\
                 & \left[2^{k+1}\bar t + 2\bar m\right]
            \end{pmatrix},
        \end{split}
    \]
    for some \(\bar t \in [0,1]\). Therefore, there are integers \(\hat n\), \(\hat m\) such that the coordinates of \(b\) are
    \[
        \begin{pmatrix}
                 & \sum_{j=0}^{k}\beta^{j} 2^{k-j}\bar t + \sum_{j=1}^{k}\beta^{j} \bar n_{j-1} +\hat n \\
                 & 2^{k+1}\bar t + \hat m
        \end{pmatrix},
    \]
    which implies the statement for $k+1$.
\end{proof}

\begin{lemma}
    For any \(k \ge 1\),
    \label{lem:most-annoying}
    \[
        \cF_A^k \gamr \cap \cF^k_A \gaml = \emptyset.
    \]
\end{lemma}

\begin{proof}
    By Lemma~\ref{lem:images}, if \(a \in   \cF_A^k \gamr \cap \cF^k_A \gaml\) for some \(k\) then there exist \(t, \bar t \in [0,1]\), \(n_j, \bar n_j , m, \bar m \in \bZ\), such that 
    \begin{multline*}
        \begin{pmatrix}
             & \beta^{k} + \sum_{j=0}^{k-1}\beta^{j} 2^{k-1-j}t + \sum_{j=0}^{k-1}\beta^{j} n_j \\
             & 2^{k}t + m
        \end{pmatrix} \\ 
        =  \begin{pmatrix}
             & \sum_{j=0}^{k-1}\beta^{j} 2^{k-1-j}\bar t + \sum_{j=0}^{k-1}\beta^{j} \bar n_j \\
             & 2^{k}\bar t + \bar m
        \end{pmatrix}.
    \end{multline*}
    The equation for the \(y\)-coordinate implies that, for some \(p \in \bZ\), 
    \[
        \bar t = t + 2^{-k}p. 
    \]
    Substituting this in the first equation we have
    \[
        \beta^{k}  + \sum_{j=0}^{k-1}\beta^{j} \left(n_j - \bar n_j -2^{-1-j} p\right) = 0.
    \]
    Since the the function \(\mathcal{P}(\beta) = \beta^{k}  + \sum_{j=0}^{k-1}\beta^{j} \left(n_j - \bar n_j -2^{-1-j} p\right)\) is a polynomial with rational coefficients and \(\beta\) is non-algebraic by assumption, the above equation is never satisfied and we conclude.
\end{proof}

In the next lemma we show that \(\gamr\) is a proper discontinuity (Definition~\ref{def:proper}) of the map \(\cF_A\).
First we construct the \(\alpha_k\), the coefficient functions.
Let \(\gamma_1 =F\gamma\), \(\gamma_{k+1} = \cF^{k}\gamma_1\), \(k \in \bN\).
Let \(\alpha_1 : \gamma_1 \to \bC\) be defined as \(\alpha_1 \equiv 1\).
We then define inductively, for all \(k\in\bN\), \(\alpha_k:\gamma_k \to \bC\) such that  \ref{it:A1} holds, i.e., for a.e. \(x\in \gamma_k\),
\begin{equation}
    \label{eq:def-A-alpha}
    \alpha_{k+1}(\cF x) = \alpha_{k}(x)\varphi(x)^{-1}  \Jac{ \gamma_{k}}{\cF}(x)^{-1}.
\end{equation}
This is well defined since \(\cF : \gamma_{k} \to \gamma_{k+1}\) is injective, at least a.e. \(x\in \gamma_{k}\).
That it is injective is due to the fact that \(\gamma_{k}\) is the union of straight lines with slope \(\widetilde \beta^{-1}_k\) and \( \widetilde \beta_{k} \neq \widetilde \beta_{k+1}\). 

\begin{lemma}
    The curve \(\gamr\) is a proper discontinuity. 
\end{lemma}

\begin{proof} 
    Assumption \textbf{\ref{it:A0}} is immediate as observed above \eqref{eq:not-equal}.

    \textbf{\ref{it:A1}:}
    The required property \eqref{eq:good-weights} holds by construction.

    \textbf{\ref{it:A2}:} 
    Suppose that, for some \(k\geq 2\), \(\gamma_k \cap (F \gamr\cup F \gaml)\) contains some curve \(\widetilde{\gamma}\). 
    This would imply that \(\widetilde{\gamma}\) is a line of slope \(\widetilde{\beta}_k^{-1}\) since \(\gamma_k\) is the union of lines of slope \(\widetilde{\beta}_k^{-1}\).
    On the other hand it also implies that \(\widetilde{\gamma}\) is a line of slope \(\widetilde{\beta}_1^{-1}\) since \(\widetilde{\gamma} \subset (F \gamr\cup F \gaml)\).
    This is a contradiction since \( \widetilde \beta_{k} \neq \widetilde \beta_{1}\).

    \textbf{\ref{it:A3}:} 
    Let \(\widetilde{\gamma} \subseteq \widehat \Gamma\) be defined as
    \[
        \widetilde \gamma = \left(\cF_A^{-1}\gamma_{k+1}\cap \bigG\right) \setminus \cF_A^k\gamr,
    \]
    and assume, for sake of contradiction, that \(\widetilde\gamma\) has  positive 1-dimensional Hausdorff measure. 
    Then, for some \(n \in \bN\), the set \(\widetilde \gamma_n = \widetilde \gamma \cap \cF_A^n (\gaml \cup \gamr) \) has positive 1-dimensional Hausdorff measure (\(\widehat \Gamma\) is the countable union of the \(\cF_A^n(\gamr\cup \gaml)\), \(n \in \bN_0\)) and it is such that \(\cF_A \widetilde \gamma_n \subseteq \cF_{A}^{k+1}\gamr \). Let us consider two separate cases. 
    First, suppose that \(n \neq k\). 
    Note that \(\cF_A \widetilde \gamma_n \subseteq \cF_A^{n+1}(\gamr \cup \gaml)\) is a union of lines with slope \(\widetilde{\beta}_{n+1}^{-1}\) while \(\cF_A^{k+1}\gamr\) is a union of lines with slope \(\widetilde{\beta}_{k+1}^{-1}\). 
    Since \(k \neq n\), we have a contradiction, because the intersection of two lines with different slopes on the torus  cannot have positive 1-dimensional Hausdorff measure. 
    Let us suppose that \(n = k\). 
    In this case, 
    \[
        \widetilde \gamma_{k} = \widetilde \gamma \cap \cF_A^n (\gaml \cup \gamr) \subseteq \cF_A^k \gaml \quad \text{and} \quad \cF_A \widetilde \gamma_{k} \subseteq \cF^{k+1}_A\gamr,
    \]
    which is not possible because of Lemma~\ref{lem:most-annoying}, concluding the proof.
\end{proof}

\subsubsection*{Comparison with upper bounds on the essential spectrum}

Consider the transfer operator associated to a piecewise smooth expanding map $\cF$ with piecewise Hölder weight \(\varphi\).
It is known~\cite{BC23,Thomine11} that the essential spectral radius of this operator acting on $BV$ is bounded above by
\begin{equation}
    \label{eq:upper-bound}
    \overline{\Lambda}(\cF, \varphi) = e^{h_{\mathrm{m}}(\cF)} \lim_{n\to \infty}\left\| \varphi_n \det D\cF^n \cdot \lambda_n^{-1}\right\|_{L^\infty}^{\frac 1n},
\end{equation}
where \(\lambda_n(x) =\inf_{\omega \in \Omega^n} \inf_{v \in \cT_x\cM: \|v\|=1} \|D_x\cF^n_{\omega}\|\) and\footnote{Note that $h_{\mathrm{m}}$ depends also on the partition $\Omega$.}
\begin{equation}
    \label{complexity}
    h_{\mathrm{m}}(\cF) = \lim_{n\to \infty} \frac{1}{n}\log \max _{x \in M} \#\left\{\omega_n \in \Omega^n \mid  \overline{\omega_{{n}}}\ni x\right\},
\end{equation}
$\Omega^n$ being the partition associated to $\cF^n$ made of cylinder sets. The quantity $h_{\mathrm{m}}(\cF)$ provides a measure of the growth of the intersection multiplicities of a discontinuity in a single point.
Notice that \(\Lambda_{BV}(\cF, \varphi, \gamr) \le  \overline{\Lambda}(\cF, \varphi)\), as it should be. 
In the affine case, such as $\cF_A$, equality holds, as shown in the following. 

\begin{lemma}
    \label{prop:beta}
    Let $\cF_A$, \(\beta\) and \(\gamma\) be as above. 
    Then
    \[
        \begin{aligned}
            \overline{\Lambda}(\cF_A, 1)                                           & = \Lambda_{BV}(\cF_A,1, \gamr)= \max\{2, \beta\},                                                   \\
            \overline{\Lambda}(\cF_A, {\left|\det D\cF_A \right|}^{-1}) & = \Lambda_{BV}(\cF_A, {\left|\det D\cF_A \right|}^{-1}, \gamr) = \min \{2, \beta\}^{-1}.
        \end{aligned}
    \]
\end{lemma}

\begin{proof}
    Following the definition \eqref{eq:Lambda}, to determine \(\Lambda_{{BV}}(\cF, 1, \gamma)\) we calculate \(\smash{|\Jac{\gamma}{\cF_A^k}|^{{1}/{k}}}\).
    For all \(\beta \in (1, \infty)\), the curve \(\gamr\) is not aligned to the direction of minimal expansion. 
    Therefore, its slope approaches the direction of the eigenvector corresponding to the maximal eigenvalue, i.e., \(\max\{2,\beta\}\). 
    In this way, one gets that \(\smash{\Jac{\gamma_k}{\cF}}\) tends to \(\max\{2,\beta\}\) as \(k \rightarrow \infty\) and we can conclude.
    Similarly, in order to determine \(\Lambda_{{BV}}(\cF, {|\det D\cF_A |}^{-1}, \gamma)\) we  calculate \(({\left|\det D\cF_A \right|}^{-1} \Jac{\gamma}{\cF_A^k})^{\frac{1}{k}}\).
    Since \(\left|\det D\cF_A \right| = 2 \beta\), \(\Lambda_{{BV}}(\cF, {\left|\det D\cF_A \right|}^{-1}, \gamma) = 2^{-1}\beta^{-1} \max\{2,\beta\} = \min \{2, \beta\}^{-1}\).

    Let us then compute the upper bounds. 
    By \cite[Lemma 1]{Buzzi97} we have \( h_{\mathrm{m}}(\cF_A) = 0\). Moreover,
    \[
        \det D\cF^n_A = (2\beta)^n, \quad \lambda_n = \min\{2,\beta\}^n.
    \]
    Therefore, by \eqref{eq:upper-bound}, \(\overline{\Lambda}(\cF_A, 1) = \max\{2,\beta \}\) whereas $\overline{\Lambda}(\cF_A, {|\det D\cF_A |}^{-1})=  \min\{2,\beta\}^{-1}$.
\end{proof}

\begin{remark}
    All indications suggest that a similar argument would imply the equality of upper and lower bounds for a larger class of piecewise affine expanding maps, whenever the discontinuity set is nowhere aligned to the direction of minimal expansion, as the ones introduced in \cite{Buzzi97}. 
    It remains to understand any possible relations between \(\Lambda_{BV}\) and the Lyapunov exponents for general two-dimensional maps with discontinuities.  
\end{remark}

\subsection{Cocycles with base which fails to be Markov}
\label{sec:cocycle}

Let \(\cM = \bT^2\) and let \(\cF_{B}: \cM \rightarrow \cM\) be defined as
\[
    \cF_{B}(x,y) = \left(T(x), S(x,y) \right),
\]
where \(T: \bT^1 \to \bT^1\) is piecewise $\cC^2$ without critical points and \(S\) is $\cC^2(\bT^2,\bT^1)$ without critical points such that \(S(x,\cdot):\bT^1 \rightarrow \bT^1\) is surjective for all \(x\). 
We assume that \(T\) is \(\cC^1\) on \(\bT^1 \setminus \{a\}\) for some \(a \in \bT^1\). 
Let \(\varphi = {|\det D\cF_B |}^{-1}\), i.e., in this example we specialise to the case where the transfer operator is the one corresponding to the SRB measure.
For convenience we use the notation \(a_{+}\), \(a_{-}\) to write
\(T^ka_{+} = \lim_{\epsilon \rightarrow 0}T^k(a+\epsilon)\) and \(T^ka_{-} = \lim_{\epsilon \rightarrow 0}T^k(a-\epsilon)\).
Finally we suppose that \(\{T^k a_{+}\}_{k=0}^{\infty}\) is a finite set of distinct points whilst \(\{T^k a_{-}\}_{k=0}^{\infty}\) is an infinite set of distinct points.

The assumptions allow us to choose \(T\) to be a \(\beta\)-map where \(\beta\) is chosen such that the map is non-Markov.
However the class of maps we consider here is much larger, there is no requirement for \(T\) to be affine. 
The assumptions imply that \(T\) is non-Markov and
\begin{equation}
    \label{eq:1D-discont}
    \begin{split}
        T^{k}a_{+} &\neq T^{j}a_{-}, \quad \forall j,k \in \bN \\
        T^{k}a_{-} &\neq T^j a_{-}, \quad j \neq k.
    \end{split}
\end{equation}
Let \(\Delta = \bigcup_{j=0}^{\infty} \{ T^j a_{+}, T^j a_{-}\}\).
In the present case we can apply the result available for one-dimensional maps with discontinuities (\cite[Lemma 3.2]{BCJ22} with \(k_0 = 1\)) and consequently
\begin{equation}
    \label{eq:1D-shift}
    T^{-1} (T^{k+1} a_{-})\cap \Delta = T^{k}a_{-}, \quad k \ge 1.
\end{equation}
Let \(\Gamma = \{a\} \times \bT^1\).
The assumptions on \(T\) and \(S\) imply that \(\cF_B\) is \(\cC^2\) on \(\bT^2 \setminus \Gamma\).
Following the notation of Section~\ref{sec:results} we consider the smooth manifold $M$ with boundary \(\partial M = \gaml \cup \gamr\), where
\[
    \quad \gamr= \{a_{-}\}\times \bT^1, \quad \gaml = \{a_{+}\} \times \bT^1.
\]
Following the notations of Definition~\ref{def:proper}, \(\gamma_1 = \{ T a_{-}\} \times \bT^1\) and, for all \(k\in\bN\), \(\gamma_{k} = \{ T^k a_{-}\} \times \bT^1\) (since \(S(x,\cdot)\) is surjective).
In this particular case we may choose the functions \(\alpha_k\) to be constant on \(\gamma_k\).
Let, for \(k \ge 2\),
\begin{equation}
    \label{eq:def-alphak}
    \alpha_{k} = (T^{k-1})'(T a_-).
\end{equation}
Note that, since \(\varphi = {|\det D\cF_B |}^{-1}\), for all \(k\in\bN\), \(x \in \gamma_{k}\), 
\begin{equation}
    \label{eq:phi-det-B}
    \varphi(x)\Jac{\gamma_{k}}{\cF_{B}}(x) = {\left|\smash{T'(T^k a_{-})}\right|}^{-1}.
\end{equation}

\begin{lemma}
    \label{prop:ex-B}
    The curve \(\gamr\) is a proper discontinuity.
\end{lemma}
\begin{proof}
    Assumption \textbf{\ref{it:A0}} is satisfied since \(T a_{+} \neq T a_{-}\) \eqref{eq:1D-discont}.

    \textbf{\ref{it:A1}:}
    Using the above \eqref{eq:phi-det-B},
    \[
        \begin{aligned}
            \alpha_{k+1} &= (T^{(k+1)-1})'(T a_-) = (T^{k-1})'(T a_-) \cdot |T'(T^k a_-)| \\
            &= \alpha_{k} \cdot \varphi(x)^{-1} \cdot \Jac{\gamma_k}{\cF_{B}}(x)^{-1},
        \end{aligned}
    \]
    so proving \ref{it:A1}. 

    \textbf{\ref{it:A2}:} 
    Again, following the notation of Definition~\ref{def:proper}, \(\Gamma_1 = \{Ta_{+},Ta_{-}\} \times \bT^1\) (using the fact that \(S(x,\cdot)\) is surjective).
    Since \(\gamma_{k} = \{ T^k a_{-}\} \times \bT^1\), the relations between the images of the discontinuity \eqref{eq:1D-discont} implies that, for all \(k\geq 2\), \(\gamma_{k} \cap \Gamma_1 = \emptyset\)

    \textbf{\ref{it:A3}:} 
    Following the notation of Section~\ref{sec:results}, \(\bigG = \bigcup\nolimits_{j\in\bN}\cF^j \Gamma_{1}\) and so (again using the fact that \(S(x,\cdot)\) is surjective),
    \[
        \widehat{\Gamma} =\left({\bigcup\nolimits_{j\in\bN}}\left\{T^j a_{-} \right\} \times \bT^1\right) \cup \left({\bigcup\nolimits_{j\in\bN}}\left\{T^j a_{+}\right\} \times \bT^1 \right).
    \]
    For \(k \in \bN\),
    \[
        \begin{split}
            \cF_{B}^{-1}\gamma_{k+1} &=\cF_{B}^{-1}\left(\left\{T^{k+1} a_{-} \right\} \times \bT^1\right) = \left\{T^{-1}\left(T^{k+1} a_{-}\right)\right\} \times \bT^1. 
        \end{split}
    \]
    However, using \eqref{eq:1D-shift},
    \[
        \begin{split}
            \left(\cF_{B}^{-1} \gamma_{k+1}\right) \cap \widehat \Gamma &=\bigcup_{j\in\bN}\left(\left\{T^{-1}(T^{k+1} a_{-}) \cap T^j (\{a_{+}\} \cup \{a_{-}\}) \right\} \times \bT^1\right) \\
            &=  \{ T^{k} a_{-}\} \times \bT^1 = \gamma_{k},
        \end{split}
    \]
    showing \ref{it:A3} and concluding the proof.
\end{proof}

\begin{remark}
Calculating the estimates \eqref{eq:Lambda}, using the formula previously obtained \eqref{eq:phi-det-B}, 
\begin{equation*}
    \Lambda_{BV}(\cF_{B},\left|\det D\cF_{B} \right|^{-1},\gamr) =\liminf_{n\to \infty}\left|(T^{n})'(a_-)\right|^{-1/n}.
\end{equation*}
Since, in this present setting, \(\len{\gamma_k}\) is uniformly bounded, \(\Lambda_{L^\infty, 2}\) will be equal to \(\Lambda_{BV}\).
On the other hand \(\Lambda_{L^\infty, 1}\) would give a worse (if there is expansion in the second coordinate) or equal (if neutral behaviour in second coordinate) bound.
The estimates on the essential spectral radius for this class of examples are equal to the estimates obtained \cite{BCJ22} for the one-dimensional map \(T\).
\end{remark}
    
\begin{remark}
    \label{rem:skew-prod}
    A broad class of interesting examples fits with the above, in particular skew products on \(\bT^2\) of the form 
    \[
        (x,y) \mapsto (T(x), y+\sigma(x)).
    \]
    If $T$ is a piecewise smooth expanding non-Markov map satisfying the assumptions above with $ T' \ge \lambda >1$ and $\sigma$ smooth (this is a partially hyperbolic system with a neutral direction) we obtain, for the transfer operator with weight as above, the estimates, $\Lambda_{BV} = \Lambda_{L^{\infty}} \ge \lambda^{-1}$.\footnote{Combining the techniques used in \cite{CL22} and \cite{BE17}, it should be possible to obtain upper bounds of the essential spectral radius for piecewise partially expanding maps, so that we could compare it with the lower bound obtained in the present paper (at least for this specific example).}
\end{remark}

\subsection{Local discontinuity}
\label{sec:local-example}

Let \(\cM = \bT^2\).
This example, denoted \(\cF_C : \cM \rightarrow \cM\), to be defined shortly, is a perturbation of a smooth uniformly expanding map in an arbitrary small neighbourhood, yet has large essential spectrum.
The only discontinuity will be an \(\epsilon\)-sized ``slit'' which is introduced by the perturbation.

Let \(f: \bT^1 \rightarrow \bT^1\) be the doubling map defined as \(f : x \mapsto 2 x \mod 1\).
For notational convenience, let \(x_0 = y_0 = \frac{1}{2}\).
For \(k \in \bN\) let \(a_k = f^{k} (x_0-\epsilon)\) and choose, once and for all, \(\epsilon>0\) small and such that \({\{a_k\}}_{k \in \bN}\) is an infinite set of distinct points, disjoint from the interval \([x_0,x_0+\epsilon]\).
Let \(J = [x_0,x_0+\epsilon] \times [y_0-\epsilon,y_0+\epsilon] \subset \bT^2\).
Fix \(\rho \in \cC^{\infty}(\bR,\bR)\) such that \(\rho(x)=1\) when \(x \leq 0\), that \(\rho(x)=0\) when \(x \geq 1\) and that, for all \(x\in \bR\), \(\rho^{\prime}(x) \leq 0\).
Let \(\rho_\epsilon(x)=\epsilon \rho(\epsilon^{-1}x)\).
Let \( T,S : \bT^2 \rightarrow \bT^2\) be defined as
\[
    \begin{split}
        T &: (x,y) \mapsto (f(x), f(y)),\\
        S &: (x,y) \mapsto 
        \begin{cases}
            (x,y)                         & \text{if \((x,y) \in \bT^2 \setminus J\)} \\
            (x - \rho_\epsilon(x-x_0), y) & \text{if \((x,y) \in J\)}
        \end{cases}
    \end{split}
\]
and hence let \(\cF_C : \cM \rightarrow \cM\) be defined as 
\[
    \cF_C = T \circ S.
\]
Note that \(\Gamma \), the set of discontinuities for \(\cF_C\), is composed of the vertical line segment
\(\{(x_0, y_0 + t): t \in [-\epsilon, \epsilon]\}\)
and the two horizontal line segments, \([x_0,x_0+\epsilon] \times \{y_0 - \epsilon\}\) and \([x_0,x_0+\epsilon] \times \{y_0 + \epsilon\}\).
As per Section~\ref{sec:results}, we consider the smooth manifold with boundary, which we denote \(M\), such that \(\cF_C : M \to \cM\) is \(\cC^\infty\).
We define
\[
      \gamr = \{x_0^{+}\} \times  [y_0-\epsilon, y_0+\epsilon], \quad \gaml = \{x_0^{-}\} \times  [y_0-\epsilon, y_0+\epsilon],
\]
where, as per the previous example, \(x_{0}^{+}\) and \(x_{0}^{-}\) correspond to the doubling of points at \(\{x_{0}\}\times  [y_0-\epsilon, y_0+\epsilon] \subset \cM\). 
Note that \(\gamr\) and \(\gaml\) are the only vertical components of \(\partial M\). 
The definition of \(a_k\) has the consequence that \(\cF_C^k(x_0^+,y) = (a_k,f^k(y))\).
Let \(\varphi = \abs{\det D\cF_C}^{-1}\) and consider, as specified in Theorem~\ref{thm:main}, the associated transfer operator \(\cL\).
Observe that \(\varphi(x,y) = \frac{1}{4}\) whenever \((x,y) \in \bT^2 \setminus J\).
Following the notation of Definition~\ref{def:proper}, \(\gamma_1 = F_{C} \gamr\) and, \(\gamma_{k+1} = \cF_{C}^{k}\gamma_1\), \(k \in \bN\).
As such, \(\gamma_k \subset \{a_k\} \times \bT^1\).

\begin{lemma}
    \label{lem:eg-C}
    The curve \(\gamr\) is a proper discontinuity.
\end{lemma}

\begin{proof}
    We choose \(\alpha_k \equiv 2^{k-1}\), \(k \in \bN\).
    Note that each \(\gamma_{k}\) is a vertical line whose first coordinate is equal to \(a_k\). 
    The choice of \(\epsilon\) means that \(\gamma_{k}\) never intersects \(J\).
    For small \(k\) these lines will have length \(2^{k}\epsilon\) and for large \(k\) they are equal to \(\{a_k\}\times \bT^1\).
    We will, in turn, check each of the required properties.

    \textbf{\ref{it:A0}}: This assumption is satisfied since \(\cF_C(x_0^+,y) = (f(x_0-\epsilon),f(y))\) whereas \(\cF_C(x_0^-,y) = (f(x_0),f(y))\) and \(f(x_0-\epsilon) \neq f(x_0)\).

    \textbf{\ref{it:A1}:}
    We see that \(\alpha_1 \equiv 1\) as required and so it remains to show that, for all \(x \in \gamma_{k}\), \(k\in\bN\),
    \begin{equation}
        \label{eq:def-alphak2}
        \alpha_{k+1}(\cF_C x) = \alpha_k(x) \varphi(x)^{-1} \Jac{\gamma_{k}}{\cF_C}(x)^{-1}.
    \end{equation}
    For every \(x\in \gamma_{k}\) we know that \(\Jac{\gamma_{k}}{\cF_C}(x) = 2\) and \(\varphi(x) = \frac{1}{4}\) so, consequently, \(\varphi(x)\Jac{\gamma_{k}}{\cF_C} \equiv \frac{1}{2}\).
    And so, since \(\alpha_k \equiv 2^{k-1}\), \ref{it:A1} is satisfied.
    
    \textbf{\ref{it:A2}:}
    We must show that  \(\gamma_{k} \cap \Gamma_1 \approxeq \emptyset \) for all \(k\geq 2\).
    We know that \(\Gamma_1\) is made up of horizontal and vertical line segments.
    However, since \(\gamma_k\) consists only of a vertical segment, we can avoid considering the horizontal segments.
    The vertical segments of \(\Gamma_1\) are contained within \(\{f(x_0-\epsilon),f(x_0)\} \times \bT^1 = \{a_1,f(x_0)\} \times \bT^1\).
    On the other hand, \(\gamma_k \subset \{a_k\} \times \bT^1\).
    Since \(\{a_k\}_k\) was assumed to be an infinite set of points, we know that \(a_n \neq a_k\), if \(n \neq k\), and \(a_{k} \neq x^+_{0}, x^{-}_{0}, 0\), for \(k \in \bN\) (otherwise \(a_{k}\) would be pre-periodic).
    Consequently the required conclusion holds.
    
    \textbf{\ref{it:A3}:}
    We must show that \(\cF_C^{-1}\gamma_{k+1} \cap \bigG\approxeq \gamma_{k}\) for all \(k\in\bN\).
    Again we can ignore the horizontal component of the discontinuity.
    Observe that \(\bigG \subset \left(\cup_{k=1}^{\infty} \{a_k\} \cup \{0\} \right) \times \bT^1\).
    Furthermore, the \(a_k\) have the property (see~\cite[Lemma 3.2]{BCJ22}) 
    \[
        \begin{split}
        &f^{-1}(a_{k+1}) \cap \left(\cup_{j=1}^{\infty} \{a_j\} \cup \{0\}\right) = a_{k} , \quad k \ge 1,
        \end{split}
    \]
    which readily implies \ref{it:A3}.
\end{proof}

\begin{lemma}
    \(\Lambda_{BV}(\cF_{C},\left|\det D\cF_{C} \right|^{-1},\gamr) = \Lambda_{L^{\infty}}(\cF_{C},\left|\det D\cF_{C} \right|^{-1},\gamr) = \frac{1}{2}\).
\end{lemma}

\begin{proof}
    Recall that, for \(x\in \gamma_{k}\), \(\varphi(x) \Jac{\gamma_{k}}{\cF_C}(x) = \frac{1}{2}\).
    Consequently \eqref{eq:Lambda},
    \[
        \Lambda_{\operatorname{BV}}
        =\liminf_{k \rightarrow \infty} \left(\smash{\inf_{x \in \gamma_{k}}} \left| \varphi_k(x) \cdot \Jac{\gamr}{\cF_C^k}(x) \right|\right)^{1/k} = \tfrac{1}{2},                       
    \]
    Additionally, for all \(k\) sufficiently large, \(\len{\gamma_{k}}=1\) and so \(\Lambda_{L^{\infty}} = \Lambda_{L^{\infty},2}\).
\end{proof}

\begin{remark}
    In the first example (Section \ref{sec:affinemap}), $\alpha_k$ can be defined inductively for any weight (see the discussion of Assumption \ref{it:A1} in Section \ref{sec:discuss}). In the second and third example (Section~\ref{sec:cocycle}, Section~\ref{sec:local-example}) we exploit the fact that \(\varphi(x)\Jac{\gamma_{k}}{\cF_{B}}(x)\) is constant on each $\gamma_k$, allowing us to define \(\alpha_k\) constant (see \eqref{eq:def-alphak} and \eqref{eq:def-alphak2}).  In the third example any weight which was constant on \(\bT^2 \setminus J\) would suffice without modifying the argument as can be seen in the proof of Lemma~\ref{lem:eg-C}.
\end{remark}


\end{document}